\documentclass[a4paper,leqno,12pt]{amsart}
\usepackage{amsmath,amsthm}
\usepackage{amssymb}
\usepackage{xspace}
\usepackage{bm}
\usepackage{amstext}
\usepackage{amsfonts}
\usepackage{graphicx}
\usepackage[mathscr]{euscript}
\usepackage{amscd}
\usepackage{latexsym}
\usepackage{amssymb}
\usepackage{enumerate}
\usepackage{mathrsfs}
\usepackage[cmtip,all]{xy}
\begin{document}
%
%
\theoremstyle{plain}
\swapnumbers
    \newtheorem{thm}[figure]{Theorem}
    \newtheorem{prop}[figure]{Proposition}
    \newtheorem{lemma}[figure]{Lemma}
    \newtheorem{keylemma}[figure]{Key Lemma}
    \newtheorem{corollary}[figure]{Corollary}
    \newtheorem{fact}[figure]{Fact}
    \newtheorem{subsec}[figure]{}
    \newtheorem*{propa}{Proposition A}
    \newtheorem*{thma}{Theorem A}
    \newtheorem*{thmb}{Theorem B}
    \newtheorem*{thmc}{Theorem C}
    \newtheorem*{thmd}{Theorem D}
\theoremstyle{definition}
    \newtheorem{defn}[figure]{Definition}
    \newtheorem{example}[figure]{Example}
    \newtheorem{examples}[figure]{Examples}
    \newtheorem{notn}[figure]{Notation}
    \newtheorem{summary}[figure]{Summary}
    \newtheorem{question}[figure]{Question}
\theoremstyle{remark}
\newtheorem{remark}[figure]{Remark}
        \newtheorem{construction}[figure]{Construction}
        \newtheorem{remarks}[figure]{Remarks}
        \newtheorem{warning}[figure]{Warning}
    \newtheorem{assume}[figure]{Assumption}
    \newtheorem{ack}[figure]{Acknowledgements}
\renewcommand{\thefigure}{\arabic{section}.\arabic{figure}}
%
%
%
\newcommand{\myfigure}[2][]
{\stepcounter{figure}\begin{equation}
     \tag{\thefigure}{#1}\vcenter{#2}\end{equation}}
\newenvironment{myeq}[1][]
{\stepcounter{figure}\begin{equation}\tag{\thefigure}{#1}}
{\end{equation}}
\newcommand{\myeqn}[2][]
{\stepcounter{figure}\begin{equation}
     \tag{\thefigure}{#1}\vcenter{#2}\end{equation}}
\newcommand{\mydiag}[2][]{\myeq[#1]\xymatrix{#2}}
\newcommand{\mydiagram}[2][]
{\stepcounter{figure}\begin{equation}
     \tag{\thefigure}{#1}\vcenter{\xymatrix{#2}}\end{equation}}
\newcommand{\mysdiag}[2][]
{\stepcounter{figure}\begin{equation}
     \tag{\thefigure}{#1}\vcenter{\xymatrix@R=15pt@C=40pt{#2}}\end{equation}}
\newcommand{\myrdiag}[2][]
{\stepcounter{figure}\begin{equation}
     \tag{\thefigure}{#1}\vcenter{\xymatrix@R=10pt@C=25pt{#2}}\end{equation}}
\newcommand{\mytdiag}[2][]
{\stepcounter{figure}\begin{equation}
     \tag{\thefigure}{#1}\vcenter{\xymatrix@R=19pt@C=50pt{#2}}\end{equation}}
\newcommand{\myudiag}[2][]
{\stepcounter{figure}\begin{equation}
     \tag{\thefigure}{#1}\vcenter{\xymatrix@R=25pt@C=50pt{#2}}\end{equation}}
\newcommand{\myvdiag}[2][]
{\stepcounter{figure}\begin{equation}
     \tag{\thefigure}{#1}\vcenter{\xymatrix@R=13pt@C=16pt{#2}}\end{equation}}
\newcommand{\mywdiag}[2][]
{\stepcounter{figure}\begin{equation}
     \tag{\thefigure}{#1}\vcenter{\xymatrix@R=5pt@C=53pt{#2}}\end{equation}}
\newcommand{\myxdiag}[2][]
{\stepcounter{figure}\begin{equation}
     \tag{\thefigure}{#1}\vcenter{\xymatrix@R=6pt@C=30pt{#2}}\end{equation}}
\newcommand{\myydiag}[2][]
{\stepcounter{figure}\begin{equation}
     \tag{\thefigure}{#1}\vcenter{\xymatrix@R=18pt@C=35pt{#2}}\end{equation}}
\newcommand{\myzdiag}[2][]
{\stepcounter{figure}\begin{equation}
     \tag{\thefigure}{#1}\vcenter{\xymatrix@R=1pt@C=33pt{#2}}\end{equation}}
%
%
\newenvironment{mysubsection}[2][]
{\begin{subsec}\begin{upshape}\begin{bfseries}{#2.}
\end{bfseries}{#1}}
{\end{upshape}\end{subsec}}
\newenvironment{mysubsect}[2][]
{\begin{subsec}\begin{upshape}\begin{bfseries}{#2\vsn.}
\end{bfseries}{#1}}
{\end{upshape}\end{subsec}}
\newcommand{\sect}{\setcounter{figure}{0}\section}
%
%
\newcommand{\wh}{\ -- \ }
\newcommand{\wwh}{-- \ }
\newcommand{\w}[2][ ]{\ \ensuremath{#2}{#1}\ }
\newcommand{\ww}[1]{\ \ensuremath{#1}}
\newcommand{\www}[2][ ]{\ensuremath{#2}{#1}\ }
\newcommand{\wwb}[1]{\ \ensuremath{(#1)}-}
\newcommand{\wb}[2][ ]{\ (\ensuremath{#2}){#1}\ }
\newcommand{\wref}[2][ ]{\ (\ref{#2}){#1}\ }
\newcommand{\wwref}[3][ ]{\ (\ref{#2})-(\ref{#3}){#1}\ }
%
%
\newcommand{\hs}{\hspace*{5 mm}}
\newcommand{\hsm}{\hspace*{2 mm}}
\newcommand{\hsn}{\hspace{2 mm}}
\newcommand{\hssp}{\hspace*{55 mm}}
\newcommand{\hsp}{\hspace*{15 mm}}
\newcommand{\vs}{\vspace{5 mm}}
\newcommand{\vsm}{\vspace{3 mm}}
\newcommand\scalemath[2]{\scalebox{#1}{\mbox{\ensuremath{\displaystyle #2}}}}
\newcommand{\xra}[1]{\xrightarrow{#1}}
\newcommand{\xepic}[1]{\xrightarrow{#1}\hspace{-5 mm}\to}
%
%
\newcommand{\ab}{\operatorname{ab}}
\newcommand{\cart}{\operatorname{cart}}
\newcommand{\Cl}{\operatorname{Cl}}
\newcommand{\coc}{\operatorname{co}}
\newcommand{\Cof}{\operatorname{Cof}}
\newcommand{\cons}{\operatorname{cons}}
\newcommand{\Coker}{\operatorname{Coker}}
\newcommand{\colim}{\operatorname{colim}}
\newcommand{\hocolim}{\operatorname{hocolim}}
\newcommand{\colimit}[1]
         {\raisebox{-1.7ex}{$\stackrel{\textstyle\colim}{\scriptstyle{#1}}$}}
\newcommand{\csk}[1]{\operatorname{csk}\sb{#1}}
\newcommand{\End}{\operatorname{End}}
\newcommand{\Ex}{\operatorname{Ex}}
\newcommand{\Fib}{\operatorname{Fib}}
\newcommand{\fin}{\operatorname{fin}}
\newcommand{\Fin}{\operatorname{Fin}}
\newcommand{\hocofib}{\operatorname{hocofib}}
\newcommand{\Ho}{\operatorname{Ho}}
\newcommand{\ho}{\operatorname{ho}}
\newcommand{\holim}{\operatorname{holim}}
\newcommand{\Hom}{\operatorname{Hom}}
\newcommand{\Id}{\operatorname{Id}}
\newcommand{\id}{\operatorname{Id}}
\newcommand{\Image}{\operatorname{Im}}
\newcommand{\inc}{\operatorname{in}}
\newcommand{\Ker}{\operatorname{Ker}}
\newcommand{\lev}{\operatorname{lev}}
\newcommand{\Mor}{\operatorname{Mor}}
\newcommand{\ND}{\operatorname{ND}}
\newcommand{\Obj}[1]{\operatorname{Obj}\,{#1}}
\newcommand{\op}{\sp{\operatorname{op}}}
\newcommand{\proj}{\operatorname{pr}}
\newcommand{\res}{\operatorname{res}}
\newcommand{\sd}{\operatorname{sd}}
\newcommand{\sk}[1]{\operatorname{sk}\sb{#1}}
%
%
\newcommand{\map}{\operatorname{map}}
\newcommand{\Map}{\operatorname{Map}}
\newcommand{\mapa}{\map\sb{\ast}}
%
%
\newcommand{\fN}{\mathfrak{N}}
\newcommand{\fC}{\mathfrak{C}}
\newcommand{\fD}{\mathfrak{D}}
\newcommand{\fd}{\mathfrak{d}}
\newcommand{\fG}{\mathfrak{G}}
\newcommand{\fH}{\mathfrak{H}}
\newcommand{\Po}[1]{P\sp{(#1)}}
%
%
\newcommand{\mA}{\mathscr{A}}
\newcommand{\mB}{\mathscr{B}}
\newcommand{\mC}{\mathscr{C}}
\newcommand{\mD}{\mathscr{D}}
\newcommand{\mK}{\mathscr{K}}
\newcommand{\mM}{\mathscr{M}}
\newcommand{\mT}{\mathscr{T}}
\newcommand{\mW}{\mathscr{W}}
\newcommand{\mX}{\mathscr{X}}
\newcommand{\mY}{\mathscr{Y}}
\newcommand{\mZ}{\mathscr{Z}}
%
%
\newcommand{\A}{\mathcal{A}}
\newcommand{\B}{\mathcal{B}}
\newcommand{\C}{\mathcal{C}}
\newcommand{\D}{\mathcal{D}}
\newcommand{\F}{\mathcal{F}}
\newcommand{\G}{\mathcal{G}}
\newcommand{\J}{\mathcal{J}}
\newcommand{\cL}{\mathcal{L}}
\newcommand{\M}{\mathcal{M}}
\newcommand{\cP}{\mathcal{P}}
\newcommand{\cS}{\mathcal{S}}
\newcommand{\U}{\mathcal{U}}
\newcommand{\W}{\mathcal{W}}
%
%
\newcommand{\CC}{\mathbb C}
\newcommand{\EE}[1]{{\mathbb E}\sb{#1}}
\newcommand{\LL}{\mathbb L}
\newcommand{\NN}{\mathbb N}
\newcommand{\QQ}{\mathbb Q}
\newcommand{\RR}{\mathbb R}
\newcommand{\bSS}{\mathbb S}
\newcommand{\ZZ}{\mathbb Z}
%
%
\newcommand{\CP}[1]{\CC\mathbf{P}\sp{#1}}
\newcommand{\bS}[1]{\mathbf{S}\sp{#1}}
\newcommand{\bT}[1]{\mathbf{T}\sb{#1}}
%
%
\newcommand{\Ab}{\mbox{\sf Ab}}
\newcommand{\AbGp}{\mbox{\sf AbGp}}
\newcommand{\Alg}{\mbox{\sf Alg}}
\newcommand{\Ch}{\mbox{\sf Ch}}
\newcommand{\DK}{\mbox{\sf DK}}
\newcommand{\Fun}{\mbox{\sf Fun}}
\newcommand{\ihom}{\mbox{\sf hom}}
\newcommand{\Kan}{\mbox{\sf Kan}}
\newcommand{\Lift}{\mbox{\sf Lift}}
%
%
\newcommand{\bA}{\mathbf{A}}
\newcommand{\bB}{\mathbf{B}}
\newcommand{\bC}{\mathbf{C}}
\newcommand{\Del}{\mathbf{\Delta}}
\newcommand{\Deln}[1]{\Del\sp{#1}}
\newcommand{\Delnk}[2]{\Deln{#1}\lo{#2}}
\newcommand{\Dop}{\Delta\op}
\newcommand{\Dres}{\Delta\sb{\res}}
\newcommand{\Drop}{\Dres\op}
\newcommand{\Dp}{\Delta\sb{+}}
\newcommand{\Dresp}{\Delta\sb{\res,+}}
\newcommand{\lra}[1]{\langle{#1}\rangle}
\newcommand{\lolr}[1]{\sb{\lra{#1}}}
\newcommand{\Drn}[1]{\Delta\lolr{#1}}
\newcommand{\Du}{\Del\sp{\bullet}}
\newcommand{\bN}{\mathbf{N}}
\newcommand{\Th}{\mathbf{Th}}
\newcommand{\Fd}{\F\sb{\bullet}}
\newcommand{\Gd}{G\sb{\bullet}}
\newcommand{\bW}{\mathbf{W}}
\newcommand{\Wd}{\bW\sb{\bullet}}
\newcommand{\Wu}{\bW\sp{\bullet}}
\newcommand{\bX}{\mathbf{X}}
\newcommand{\bx}{\mathbf{x}}
\newcommand{\bY}{\mathbf{Y}}
\newcommand{\bZ}{\mathbf{Z}}
%
%
\newcommand{\mBr}{\mathbf{Br}}
\newcommand{\BrCat}{\mbox{\sf BrCat}}
\newcommand{\Cat}{\mbox{\sf Cat}}
\newcommand{\cCat}{\mbox{\sf cCat}}
\newcommand{\sCat}{\mbox{\sf sCat}}
\newcommand{\FinCub}{\mbox{\sf FinCub}}
\newcommand{\FreeBrCat}{\mbox{\sf FreeBrCat}}
\newcommand{\mFreeBrCat}{\underline{\FreeBrCat}}
\newcommand{\FreeMon}{\mbox{\sf FreeMon}}
\newcommand{\mFreeMon}{\underline{\FreeMon}}
\newcommand{\mMnd}{\mathbf{Mnd}}
\newcommand{\Model}{\mbox{\sf Model}}
\newcommand{\mMon}{\mathbf{Mon}}
\newcommand{\Set}{\mbox{\sf Set}}
\newcommand{\Seta}{\Set\sb{\ast}}
\newcommand{\sSeta}{\sSet\sb{\ast}}
\newcommand{\sSet}{\Set\sb{\Delta}}
\newcommand{\sSetm}{\Set_\Delta }
\newcommand{\Sp}{\mbox{\sf Sp}}
\newcommand{\Mod}{\mbox{\sf Mod}}
\newcommand{\Top}{\mbox{\sf Top}}
\newcommand{\Tz}{\Top\sb{0}}
\newcommand{\Topa}{\Top\sb{\ast}}
\newcommand{\Tw}{\mbox{\sf Tw}}
\newcommand{\Un}{\mbox{\sf Un }}
\newcommand{\St}{\mbox{\sf St }}

\newcommand{\icat}{\Cat\sb{\infty}}
\newcommand{\cSet}{\mbox{\sf cSet}}
%
%
\newcommand{\sq}[1]{\square\sp{#1}}
%
%
\newcommand{\bbo}{\mathbf{1}}
\newcommand{\bbt}{\mathbf{2}}
\newcommand{\bbj}{[\mathbf{j}]}
\newcommand{\bk}{[\mathbf{k}]}
\newcommand{\bkm}{[\mathbf{k}-\bbo]}
\newcommand{\bmm}{[\mathbf{m}]}
\newcommand{\bbn}{\mathbf{n}}
\newcommand{\bn}{[\bbn]}
\newcommand{\bnp}{[\mathbf{n}+\bbo]}
\newcommand{\bon}{[\mathbf{1}]}
\newcommand{\bone}[1]{[\mathbf{1}]\sp{#1}}
\newcommand{\bne}{[\mathbf{1}]\sb{e}}
\newcommand{\bnm}{[\mathbf{n}-\mathbf{m}]}
\newcommand{\bnmp}{[\mathbf{n}-\mathbf{m}+\bbo]}
\newcommand{\bnrp}{[\mathbf{n}-\mathbf{r}+\bbo]}
\newcommand{\br}{[\mathbf{r}]}
\newcommand{\bmo}{[-\bbo]}
%
%
\newcommand{\iot}[1]{\iota\sb{#1}}
\newcommand{\vare}{\varepsilon}
\newcommand{\var}[1]{\vare\sb{#1}}
\newcommand{\varp}[1]{\vare'\sb{#1}}
\newcommand{\gam}[1]{\gamma\lo{#1}}

\title{Andr\'{e}-Quillen cohomology and $k$-invariants of simplicial categories}
%
\author[D.~Blanc]{David Blanc}
\address{Department of Mathematics,
         Faculty of Natural Sciences,
         University of Haifa,
         PO Box 3338
         3498838 Haifa,
         Israel}
\email{blanc@math.haifa.ac.il}
%
\author[N.J.\ Meadows]{Nicholas J.\ Meadows}
%
\address{Department of Mathematics,
         University of Bologna,
         Piazza di Porto San Donato, 
         Bologna 40126, 
         Italy}
\email{njmead81118@gmail.com}
\date{\today}
\makeatletter
\@namedef{subjclassname@2020}{%
  \textup{2020} Mathematics Subject Classification}
\makeatother
\subjclass[2020]{Primary: 55S45; Secondary: 18N60, 18D20, 18G90}
\keywords{Cohomology, simplicial category, $\infty$-category, $k$-invariant}

\begin{abstract}
Exploiting the Harpaz-Nuiten-Prasma interpretation of the Dwyer-Kan-Smith cohomology
of a simplicial category $\mX$, we obtain a cochain complex for the
Andr\'{e}-Quillen cohomology groups in which the $k$-invariants for $\mX$
take value. Given a map of simplicial categories \w{\phi:\mY\to\Po{n-1}\mX}
into a Postnikov section of $\mX$, we use a homotopy colimit decomposition
of $\mY$ to study the obstruction to lifting $\phi$ to \w[.]{\Po{n}\mX} In particular,
an explicit description of this obstruction for the boundary of a cube can be
used to recover various higher homotopy invariants of $\mX$.
\end{abstract}

\maketitle
\setcounter{section}{-1}

%
%
\sect{Introduction}
\label{cint}

The Postnikov tower \w{(\Po{n}X)\sb{n=0}\sp{\infty}} of a (path-connected) topological
space $X$ is a useful way to decompose it into simpler pieces,
but in order to pass from \w{\Po{n}X} to \w[,]{\Po{n+1}X}
in addition to knowing \w[,]{\pi\sb{n+1}X} we also need ``gluing information''
provided by the $n$-th $k$-invariant
\w{k\sb{n}\in H\sp{n+2}(\Po{n}X;\pi\sb{n+1}X)} (with twisted coefficients, if
$X$ is not simply connected). This is a rather opaque invariant which is hard to define,
and harder to understand conceptually.

Because the functors \w{\Po{n}} are monoidal, we may apply them objectwise to
the mapping spaces of a topologically or simplicially enriched category $\mX$,
obtaining a similar tower
\begin{myeq}[\label{eqpostnikov}]
\dotsc\to\Po{n+1}\mX\to\Po{n}\mX\to\Po{n-1}\mX\to\dotsc\Po{0}\mX
\end{myeq}
of such enriched categories, whose inverse limit is equivalent to $\mX$.

Again, there are $k$-invariants allowing one to recover \w{\Po{n+1}\mX} from
\w[,]{\Po{n}\mX} taking value in the cohomology theory
of simplicial categories defined by Dwyer, Kan, and Smith in
\cite{DKSmitO}. Our purpose here is to explain how the re-interpretation
of this cohomology theory as a form of Andr\'{e}-Quillen cohomology, given by
Harpaz, Nuiten, and Prasma in \cite{HNPrasA} yields a more useful
description of the $k$-invariants for $\mX$, which allows us to recover
certain higher-order homotopy invariants in $\mX$.

The observation underlying our approach is the fact that commuting \emph{cubes} in
an \wwb{\infty,1}category, rather than simplices, best encode higher order
information, because they exhibit decompositions of morphisms in more than one way.
If for any $n$ we denote by \w{\square\sp{n}} the simplicial category representing
the $n$-cube (see \S \ref{sccube}), our main technical result shows that these
suffice to describe simplicial categories:

\begin{thma}
For any simplicial category $\mX$, the corresponding
quasi-category $X$ is the homotopy colimit of its non-degenerate cubes
\w{\sq{n} \to X} \wb[.]{n\geq 0}
\end{thma}
\noindent See Theorem \ref{thm5.8} below\vsm.

This is then used to obtain an explicit formula for the cochain complex computing
the Andr\'{e}-Quillen cohomology groups housing the $k$-invariants:

\begin{thmb}
For every fibrant and cofibrant simplicial category $\mX$,
\w{H\sp{\ast}\sb{Q}(\mX,\pi\sb{n}\mX)} is the cohomology of
the limit over all cubes \w{\phi:\sq{k}\to\mX} of the cochain complex
$$
  \bigoplus\sb{\ell([x\sb{0}, y\sb{0}]) = 0}
  \pi\sb{n}(\Map\sb{\mX}(x\sb{0}, y\sb{0}), \phi\sb{0}) \to
  \cdots \to \bigoplus\sb{l[x\sb{n+1}, y\sb{n+1}] = n+1}
  \pi\sb{n}(\Map\sb{\mX}(x\sb{n+1}, y\sb{n+1}), \phi\sb{n+1})
$$
\noindent with \w{[x\sb{k},y\sb{k}]} ranging over all intervals of the cube.
\end{thmb}
\noindent See Theorem \ref{thm7.3} below\vsm.

One of the main uses of the $k$-invariants is to provide an obstruction class
$\beta$ in \w{H\sp{\ast}(\mY,\phi\sp{\ast}\pi\sb{n}\mX)} for the existence
of liftings of the form
\mytdiag[\label{eqppost}]{
      &&     \Po{n}\mX \ar[d]  \\
\mY \ar@{.>}[urr] \ar[rr]\sp{\phi} &&   \Po{n-1}\mX
}
in simplicial categories (see \wref{eqpost} below).

Given a map of simplicial categories \w[,]{\phi:\partial\square\sp{n}\to\Po{n-1}\mX}
let $M$ be the mapping space in $\mX$ from the source to the target of the cube.
We show that by combining all the facets of the cube, we obtain a map
\w{\beta:\bS{n-2}\to M} (see Proposition \ref{pperm}) and show:

\begin{thmc}
The obstruction class for lifting $\phi$ to \w{\Po{n}\mX} is represented by
$\beta$.
\end{thmc}
\noindent See Theorem \ref{tpbclass} below\vsm.

In principle, the successive $k$-invariants of a simplicial category $\mX$
should contain all higher order information about $\mX$ (beyond the primary
information contained in \w{\pi\sb{n}\mX} for each \w[).]{n\geq 0}
Section \ref{ckihs} shows how the information contained in Theorem C can be used
in a few simple examples to derive such higher homotopy invariants, including:
\begin{enumerate}
\item differentials in the spectral sequence of a simplicial object in $\mX$;
\item higher order homotopy operations;
\item the classification of certain diagrams in $\mX$.
\end{enumerate}

Our final section provides another example, by addressing the question of when
a loop space can be given a compatible \ww{\EE{2}}-structure, using the
description of algebras over an operad as models of certain \ww{\Fin}-\emph{theories}
(see \cite{HMeadS}).

\begin{notn}\label{snac}
Let \w{\Dp} denote the category of finite ordered sets and
order-preserving maps, with objects \w{\bn=(0<1<\dotsc<n)} for \w[,]{n\geq -1}
$\Delta$ the full subcategory with \w[,]{n\geq 0}
\w{\Dresp} the wide subcategory of \w{\Dp} with only monic maps, 
and \w{\Drn{n,r}} its full subcategory with objects \w{\bk} that satisfy 
\w[.]{r\leq k\leq n+1} 
A \emph{simplicial object} in a category $\C$ is a functor
\w[,]{\Dop\to\C} and the category of all such is denoted by \w[.]{\C\sp{\Dop}}

The category of simplicial categories (i.e., those enriched in
\w[)]{\Set\sp{\Dop}} will be denoted by \w[,]{\sCat} with objects denoted by
$\mX$, $\mY$, etc. The simplicial enrichment
in $\mX$ is written \w[,]{\Map\sb{\mX}(x,y)} and \emph{fibrant} simplicial categories
in Bergner's model structure (see \cite{Bergner1}) are those for which
\w{\Map\sb{\mX}(-,-)} takes values in Kan complexes.

A useful form of cofibrant replacement for \w{\mX\in\sCat} is given by the
Dwyer-Kan resolution \w{\DK(\mX)} (see \cite[\S 2]{DKanL}).
The simplicial category of simplicial sets, enriched by internal \w[,]{\ihom}
will be denoted by \w[.]{\sSetm\in\sCat}
The homotopy category of a model category $\M$ is written \w[.]{\Ho(\M)} Given two objects
\w[,]{x, y \in \M} we write \w{\Map\sp{h}\sb{\M}(x,y)} for the Dwyer-Kan mapping space
between them (see \cite{DKanL}).

\emph{Quasi-categories} are fibrant objects for the Joyal model structure
on \w[,]{\Set\sp{\Dop}} with Joyal equivalences as the weak equivalences
(see \cite[Theorem 2.2.0.1]{LurieHTT}). We will use $X$, $Y$, and so on, to denote
quasi-categories. For objects ($0$-simplices) \w[,]{x,y\in X} \w{C\sb{/x}} is the
slice quasi-category over $x$, \w{C\sb{x/}} the coslice quasi-category, and
\w{\Map\sb{X}(x,y)} the space of maps between them (using any of the equivalent
models from \cite[Section 1.2.2]{LurieHTT}). The \emph{functor quasi-category}
between two quasi-categories, given by the internal \w{\ihom} in simplicial sets,
is denoted by \w[.]{\Fun(X,Y)}  The quasi-category of spaces
(see \cite[Definition 1.2.6.1]{LurieHTT}) is denoted by $\cS$, and that
of quasi-categories (see \cite[Definition 3.0.0.1]{LurieHTT}) by \w[.]{\icat}

The adjoint functors
$$
\fC:\Set\sp{\Dop}\leftrightarrows\sCat:\fN~,
$$
\noindent with $\fN$ the \emph{homotopy coherent nerve}, and $\fC$
described explicitly in \cite[Example 1.1.5.9]{LurieHTT}, constitute a Quillen
equivalence between the model structures of Joyal and Bergner, by
\cite[Theorem 2.2.5.1]{LurieHTT}.
\end{notn}

\begin{mysubsection}{Organization}
\label{sorg}
Section \ref{csmc} reviews the stabilization of model categories and the theory
of spectral Andr\'{e}-Quillen cohomology from \cite{HNPrasA} and \cite{HNPrasT}.
Section \ref{ccsms} provides some background on cubical sets and their model structures.
In Sections \ref{chcdcs} and \ref{chcdcscc} we show how a quasi-category
decomposes as a homotopy colimit of $n$-cubes.
In Section \ref{csaqcc} we define chain complexes computing the
Andr\'{e}-Quillen cohomology of an $n$-cube, and thence that of a
simplicial category.
In Section \ref{ccccc} we describe the lifting obstructions for the boundary of a
cube in a simplicial category, and in Section \ref{ckihs} we show
how these $k$-invariants encode certain elements of higher structure.
Finally, Section \ref{cotrea} describes an obstruction theory for extending a given
\ww{\EE{1}}-algebra structure to an \ww{\EE{2}}-algebra.
\end{mysubsection}

\begin{ack}
We would like to thank the referee for his or her helpful comments.
The second author was supported by the Starting Grant 101077154
``Definable Algebraic Topology'' from the European Research Council awarded to
Martino Lupini.
\end{ack}

%
%
\sect{Spectral Andre-Quillen Cohomology}
\label{csmc}

We first review the stabilization of a model category $\M$ and the
cotangent complex of an object in $\M$ (from \cite{HNPrasT}), which
are needed to define spectral Andr\'{e}-Quillen cohomology
(see \cite{HNPrasA}). This has an advantage over Quillen's original definition
(which assumes the existence of a model structure on abelian group objects),
and behaves better with respect to Quillen equivalences.

An $\Omega$-\emph{spectrum} is a diagram \w{X : \NN \times \NN \to \M}
such that \w{X\sb{m, n} := X(m, n)} is weakly contractible for each $m \neq n$ and for
which each diagonal square:
$$
\xymatrix@R=10pt@C=20pt{
X\sb{n, n} \ar[d] \ar[r] & X\sb{n, n+1} \ar[d] \\
 \ar[r] X\sb{n+1, n} \ar[r] & X\sb{n+1, n+1}
}
$$
is a homotopy pullback.

Given a model category $\M$, the left Bousfield localization of the injective model
structure \w[,]{\M\sp{\NN \times \NN}} in which the fibrant objects are Reedy fibrant
$\Omega$-spectra, is denoted by \w{\Sp(\M)} (provided the localization exists).

\begin{thm}\label{thm1.1}
For $\M$ a left proper combinatorial pointed model category, \w{\Sp(\M)} exists.
\end{thm}

\begin{proof}
See \cite[Lemma 2.1.6. Corollary 2.1.7]{HNPrasT}.
\end{proof}

The diagonals determine equivalences \w[,]{X\sb{n, n} \simeq \Omega X\sb{n+1, n+1}}
so we call \w{\Sp(\M)} the \emph{category of spectrum objects in $\M$}, or simply
the \emph{stabilization} of $\M$. The suspension functor
\w{\Sigma\sp{n} :\Sp(\M) \to\Sp(\M)} (for \w[)]{n\geq 0}
is given by \w[,]{\Sigma\sp{n}(X\sb{\bullet, \bullet})\sb{k, l} := X\sb{k+n, l+n}}
and we have a Quillen pair
$$
\Sigma\sp{\infty} : \M \leftrightarrows \Sp(\M) :  \Omega\sp{\infty}~,
$$
\noindent where \w{\Sigma\sp{\infty}(X)} is the constant diagram on $X$,
and \w[.]{\Omega\sp{\infty}(X\sb{\bullet, \bullet})=X\sb{0, 0}}

For an unpointed  model category  $\M$, we have composite adjunctions
$$
\Sigma\sp{\infty}\sb{+} : \M \leftrightarrows \M\sb{\ast}
\leftrightarrows\Sp(\M\sb{\ast})~:  \Omega\sp{\infty}\sb{+}
$$
where \w{\M\sb{\ast} := \M\sb{\ast/}} is the `pointification' of $\M$
(the slice model category over the terminal object). Given \w[,]{A \in \M}
we write \w{\M\sb{A//A}} for \w[.]{(M\sb{/A})\sb{\Id\sb{A}/}}

For each \w[,]{n\in\NN} the category \w{\Sp(\M)} is equipped with
\emph{shift functors}
$$
(-)[-n]\colon\Sp(\M) \leftrightarrows\Sp(\M)\colon(-)[n]~,
$$
\noindent given by \w{X[n]\sb{\bullet, \bullet}:=X\sb{\bullet+n, \bullet+n}} and
\w[.]{X[-n]\sb{\bullet,\bullet}:=X\sb{\bullet-n,\bullet-n}}
These shift functors forms a Quillen pair, since \w{[-n]} preserves cofibrations
and levelwise weak equivalences, and \w{[n]} preserves $\Omega$-spectra.

We have a similar notion of \emph{stabilization} \w{\Sp(X)} of a quasi-category $X$
(see \cite[Proposition 1.4.2.22]{LurieHA}). One can identify the
stabilization of a model category with that of the associated quasi-category
(see \cite[Proposition 3.3.2]{HNPrasT}).

\begin{defn}\label{def1.2}
If $A$ is an object in a model category $\M$, then \w[,]{\mT\sb{A}=\mT\sb{A}(\M)}
the \emph{tangent model category} at $A$, is defined to be \w[.]{\Sp(\M\sb{A//A})}

The \emph{cotangent complex} \w{\cL\sb{A}} of $\M$ at $A$ is the image of
\w{A \xra{\id} A} under the left derived functor
\w[.]{\LL\Sigma\sp{\infty}\sb{+} : \M\sb{/A} \to\Sp(\M\sb{A//A})}
The \emph{relative cotangent complex} \w{\cL\sb{A/B}} of \w{f:B\to A}
is the homotopy cofiber of \w{\LL\Sigma\sp{\infty}\sb{+}(f):\cL\sb{B}\to \cL\sb{A}}
in \w[.]{\Sp(\M\sb{A//A})}
\end{defn}

Now for each \w{f : A \to B} in $\M$, we have an adjunction
\begin{myeq}[\label{con1.4}]
f\sb{!}~\colon~\M\sb{A//A}~\leftrightarrows~\M\sb{B//B}~\colon~f\sp{\ast}~,
\end{myeq}
\noindent with \w{f\sb{!}(A\to X \to A)} given by
\w[,]{B=A\coprod\sb{A}B\to X\coprod\sb{A} B \to A \coprod\sb{A} B = B }
and \w{f\sp{\ast}(A \to Y \to A)} given by
\w[.]{A=A\times\sb{B} B\to X\times\sb{B}A\to A\times\sb{B} B=A}
This induces an adjunction
\begin{myeq}[\label{con1.3}]
f\sb{!}~:~\mT\sb{A}~\leftrightarrows~\mT\sb{B}~:~f\sp{\ast}~.
\end{myeq}

\begin{remark}\label{rmk1.4}
By \cite[Remark 3.3.5]{HNPrasT}, we can identify the stabilization of $\M\sb{A//A}$
with the model category whose objects are \emph{reduced excisive functors}
\w[,]{\Fin\sb{\ast} \to \M\sb{A//A}} where  \w{\Fin\sb{\ast}} is the
category of finite pointed simplicial sets. This has the structure of a very special
$\Gamma$-space in the sense of Segal, and thus an object of \w{\Sp(\M\sb{A//A})} gives
rise to an \ww{\EE{\infty}}-monoid structure on $A$.
\end{remark}

\begin{defn}\label{def1.5}
Suppose that $\M$ is a model category whose underlying category $M$ has a $0$-object $0_M$.
A \emph{weak $0$-object} in $\M$, is a object which is weakly equivalent to $0\sb{M}$.
A \emph{weak $0$-section} is a map \w{f:0\to M[n]} in \w{\Sp(\M\sb{A//A})}
whose domain is a weak $0$-object. We call the induced map \w{\Omega\sp{\infty}\sb{+}(f)}
a \emph{weak $0$-section} of \w{\Omega\sp{\infty}\sb{+}(M[n])} in $\M$.
If \w{\Sp(\M)} exists and $X$ is a fibrant object of $\M$, then a \emph{small extension}
of $X$ is a homotopy pullback diagram:
$$
\xymatrix@R=12pt@C=20pt{
X\sb{\alpha} \ar[r] \ar[d] & \Omega\sp{\infty}\sb{+}(0) \ar[d] \\
X \ar[r] & \Omega\sp{\infty}\sb{+}(M[1])~.
}
$$
where the right vertical map is \w{\Omega\sp{\infty}\sb{+}(f)} for some weak
$0$-section $f$.
\end{defn}

\begin{example}\label{exam1.6}
In the Kan model structure on \w[,]{\Set\sp{\Dop}} \w{(\Set\sp{\Dop})\sb{/X}}
can be identified with the quasi-category \w{\cS\sb{/X}} and
the stabilization can be identified with the functor quasi-category
\w[.]{\Fun(X, \Sp(\cS))} That is, small extensions are parametrized spectra.
In the case of a constant functor, the \w{\EE{\infty}} group structure splits as
\w[.]{\Omega\sp{\infty}\sb{+}(M) = \Omega\sp{\infty}\sb{+}M\sb{0,0}\times X}

Thus, we can think of a small extension in this situation as a principal
$\infty$-bundle or torsor (see \cite[Definition 2.16 and Proposition 3.8]{NSSteP}).
See \S \ref{skinv} below.
 \end{example}

\begin{defn}\label{def2.1}
For \w[,]{M \in\Sp(\M\sb{A//A})} the $n$th \emph{spectral Andr\'{e}-Quillen cohomology
group} for \w{A\in\M} \emph{with coefficients in $M$} is defined to be
$$
H\sp{n}\sb{Q}(A, M) := \pi\sb{0}\Map\sp{h}(\cL\sb{A},\Sigma\sp{n}M)~
$$
\noindent (see \S \ref{snac}). The \emph{relative spectral Andr\'{e}-Quillen
cohomology group} for \w{f:B\to A} is
$$
H\sp{n}\sb{Q}(A,B; M) := \pi\sb{0}\Map\sp{h}(\cL\sb{A/B}, \Sigma\sp{n}M)~.
$$
\end{defn}

\begin{mysubsection}{Classical Andr\'{e}-Quillen cohomology}
\label{scaqc}
 The category \w{\Ab(\M\sb{/A})} of abelian
group objects in \w[,]{\M\sb{/A}} for fixed \w[,]{A\in\M} has a model structure induced
by the free-forgetful adjunction \w[.]{\F:\M\sb{/A}\leftrightarrows\Ab(\M\sb{/A}):\U}
In this situation we define the $n$-th \emph{Quillen cohomology group} of $A$ with
coefficients in \w{M \in\Ab(\M\sb{/A})} to be
$$
H\sp{n}\sb{\Cl, Q}(A, M)~:=~\pi\sb{0}\Map\sp{h}(\LL\F(A), \Sigma\sp{n}(M))~.
$$

By \cite[Proposition 2.3]{HNPrasA}, we have isomorphisms
\w[,]{H\sp{n}\sb{\Cl, Q}(X, M)\cong H\sp{n}\sb{Q}(X, H(M))}
where $H$ is the composite
\w[.]{\Ab(\M\sb{/A}) \xra{\Sigma\sp{\infty}}\Sp(\Ab(\M\sb{/A}))
  \xra{\Sp(\U)}\Sp(\M\sb{A//A})}
\end{mysubsection}

\begin{remark}\label{exam2.2}
In order to facilitate the comparison with \cite{HNPrasA}, the indexing of the
cohomology groups we use here is not that of Quillen in \cite[II, \S 5]{QuiH} or
Dwyer, Kan, and Smith in \cite{DKSmitO}. In particular, it does not match that of
the classical $k$-invariants \w{k\sb{n}\in H\sp{n+2}(\Po{n}X;\pi\sb{n+1}(X,x\sb{0}))}
for a pointed space \w[.]{(X,x\sb{0})}
\end{remark}

\begin{prop}\label{thm2.4}
Let $\M$ be a combinatorial left proper model category and \w{M \in \M} a
fibrant $\Omega$-spectrum object. If
$$
\xymatrix@R=12pt@C=20pt{
X \ar[d] \ar[r] & \Omega\sp{\infty}(0)\ar[d] \\
Y \ar[r]\sb{g} & \Omega\sp{\infty}(M[1])
}
$$
is a small extension, there is a natural class
\w{[\beta] \in H\sp{1}\sb{Q}(A,B;g\sp{\ast}M[1])} serving as the obstruction to the
existence of the lift in
$$
\xymatrix@R=14pt@C=20pt{
B \ar[d]\sb{f} \ar[r] & X \ar[d] \ar[r] & \Omega\sp{\infty}(0) \ar[d] \\
A \ar[r]\sb{g} \ar@{.>}[ur] & Y \ar[r]_<<<<{\alpha} & \Omega\sp{\infty}(M[1])~.
}
$$
\end{prop}

\begin{proof}[Sketch of Proof]
This follows from \cite[Section 2.6]{HNPrasA}.
The existence of the lift is equivalent to the map \w{ \alpha \circ g} being
nullhomotopic. By adjunction, this determines a map $\beta$ from the homotopy cofiber
of \w[,]{\cL\sb{B}\to\cL\sb{A}} which can be identified with the relative
cotangent complex.
\end{proof}

\begin{thm}\label{thmxxx}
If $\mX$ is a fibrant simplicial category, \w{\mT\sb{\mX}(\sCat)} can be
identified with the stabilization of
\w[,]{\Fun(\mX \times \mX\sp{\op}, \sSet)\sb{\Map\sb{\mX}//\Map\sb{\mX}}}
with the model structure of \ww{\sSet}-enriched functors between \w{\mX\times\mX\sp{\op}}
and \w[.]{\sSet} Under this identification, the adjunction of \wref{con1.4} induced
by a map \w{\mX \to \mY} can be identified with the stabilization of the
restriction/left Kan extension adjunction:
$$
f\sb{!} :  \Fun(\mX \times \mX\sp{\op}, \sSet)\sb{\Map\sb{\mX}//
 \Map\sb{\mX}} \leftrightarrows
\Fun(\mY \times \mY\sp{\op}, \sSet)\sb{\Map\sb{\mY}//\Map\sb{\mY}} : f\sp{\ast}~.
$$
\end{thm}

\begin{proof}
The first statement follows from \cite[Theorem 3.1.14]{HNPrasA}, and the fact that
the monoidal structure for simplicial sets is cartesian. The second follows from
\cite[Equation 3.2.6]{HNPrasA}.
\end{proof}

\begin{defn}\label{def2.5}
Suppose that $C$ is a category. The \emph{twisted arrow category} of $C$, denoted $\Tw(C)$, is the category whose objects are arrow of $C$ and whose morphisms from $h$ to $i$ are commutative diagrams:
$$
\xymatrix
{
x\ar[d]_h  & \ar[l]\sb{f} \ar[d]\sb{i} z \\
y \ar[r]_g & w
}
$$
in $C$.

Denote by $\epsilon$ the functor \w{\Dop\to\Dop}
$$
[n] \mapsto [n]\ast[n]\op \cong [2n+1]
$$
from \cite[Definition 5.2.1.1]{LurieHA}. Given a simplicial set $X$, we denote by
\w{\Tw(X)} the simplicial set \w[,]{\epsilon\sp{\ast}(X)} so that
\w[.]{\epsilon\sp{\ast}(X)\sb{n} = X\sb{2n+1}}
and call this the \emph{twisted arrow category} of $X$
 By \cite[Remark 5.6.4]{CisiH}, we have an identity $N\Tw(C) = \Tw(N(C))$, justifying the coincidence of terminology.
\end{defn}

\begin{thm}\label{thm2.6}
If $\mX$ is a fibrant simplicial category, the underlying $\infty$-category of
\w{\mT\sb{\mX}} can be identified with the $\infty$-category of functors
\w{\Tw(\fN(\mX)) \to\Sp(\cS)} into the $\infty$-category of spectra.
The cotangent complex of $\mX$ can be identified with the constant functor
on the desuspension of the sphere spectrum.
Under these identifications, the adjunction of \ref{con1.3} can be identified with
the restriction/left Kan extension adjunction
$$
f\sb{!} :  \Fun(\Tw(\fN(\mX)), \Sp) \leftrightarrows
\Fun(\Tw(\fN(\mY)), \Sp) : f\sp{\ast}~.
$$
\end{thm}

\begin{proof}
This is \cite[Corollary 3.3.1]{HNPrasA} combined with the last paragraph of
\cite[pg. 792]{HNPrasA}
\end{proof}

\begin{lemma}\label{lemmapping}
If the quasi-category $X$ is the homotopy limit of a diagram of quasi-categories
\w[,]{\Phi: I \to \sSet} with
\w{i\sb{j} : X \to \Phi(j)} the structure map for \w[,]{j \in I}
then for each $x$, $y$, and $z$ in $X$:
$$
\Map\sb{X}(x, y) = \underset{j \in I}{\holim} \Map\sb{\Phi(j)}(i\sb{j}(x), i\sb{j}(y))
$$
(in the Kan model structure).
\end{lemma}

\begin{proof}
The  model \w{\Hom\sb{X}(x, y)} of the mapping space in a quasi-category
from \cite[Section 1.2.2]{LurieHTT}  is a homotopy pullback in the Joyal model
structure by \cite[Proposition A.2.4.4]{LurieHTT}. Thus, it commutes with (Joyal)
homotopy limits. Homotopy limits of Kan complexes in the Joyal model structure are
Kan homotopy limits as well by \cite[Theorem 6.22 and Proposition 6.26]{JoyQA},
so the statement follows.
\end{proof}

The following result allows one to decompose the spectral Andr\'{e}-Quillen cohomology
of a simplicially enriched category in terms of simpler categories:

\begin{thm}\label{thm2.7}
Consider a lifting problem:
$$
\xymatrix@R=12pt@C=25pt{
\emptyset \ar[d] \ar[r] & \mX \ar[d] \ar[r] & \Omega\sp{\infty}(0) \ar[d] \\
\mB \ar[r]\sb{g} \ar@{.>}[ur] & \mY \ar[r]_<<<<{\alpha} & \Omega\sp{\infty}(M[1])~,
}
$$
\noindent where $\mB$, $\mX$ and $\mY$ are fibrant simplicial categories, with
\w[.]{B = \fN(\mB)} Let $\F$ be the object of \w{\Fun(\Tw(B),\Sp)}
corresponding to \w{g\sp{\ast}M[1]} under Theorem \ref{thm2.6}.  Suppose $I$ is
a filtered category, and $\mB$ is the homotopy colimit of \w{\Phi\colon I\to\sCat}
given by \w[.]{\{\phi\sb{i}\colon\Phi(i)\to\mB\}\sb{i \in I}}
and let \w[.]{\Psi = \fN(\Phi)}
Then
\begin{myeq}[\label{eqlimit}]
\Map\sb{\Fun(\Tw(B),\Sp)}(\cL\sb{B},\F)~=~
\underset{i \in I}{\holim}\,
\Map\sb{\Fun(\Tw(\Psi(i)),\Sp)}(\cL\sb{\Psi(i)},\ \fN(\phi\sb{i})\sp{\ast}\F)~.
\end{myeq}
\noindent Under this equivalence, the obstruction class
\w{\beta \in \Map\sb{\Fun(B,\Sp)}(\cL\sb{B},\F)} of Proposition
\ref{thm2.4} can be identified with the homotopy limit of the obstruction classes
$$
\beta\sb{\Psi(i)} \in
  \Map\sb{\Fun(\Tw(\Psi(i)),\Sp)}(\cL\sb{\Psi(i)}, \fN(\phi\sb{i})\sp{\ast}\F)~.
$$
\end{thm}

\begin{proof}
The twisted arrow category construction commutes with filtered colimits, and filtered
colimits are homotopy colimits. Since \w{\Fun(-, \Sp)} sends homotopy colimits to
homotopy limits, \w{\Fun(\Tw(B),\Sp)} is a homotopy limit of the functor
\w{\fN(I)\to\icat} given by
$$
i \mapsto\Fun(\Tw(\Psi(i)),\Sp)
$$
The first statement now follows from Lemma \ref{lemmapping}.
The second statement then follows from the naturality of the obstruction class
of Proposition \ref{thm2.4}, and Theorem \ref{thm2.6}.
\end{proof}

\begin{mysubsection}{The $k$-invariants of a simplicial category}
\label{skinv}
Our main example of small extension (\S \ref{exam1.6}) is the $n$-th
\emph{Postnikov section} \w{\Po{n}\mX} of a fibrant simplicial category $\mX$.
This is obtained by applying the \wwb{n+1}coskeleton functor
\w{\csk{n+1}} (right adjoint to the \wwb{n+1} skeleton functor \w[)]{\sk{n+1}}
to each simplicial mapping space of $\mX$.

Let \w{\Pi\sb{1}X} denote the fundamental groupoid of a space $X$ (see \cite[I.8]{GJarS}).
For any \w[,]{m\geq 2} the local system \w{\Pi\sb{1}X\to\AbGp} which sends
$x$ to \w{\pi\sb{n}(X, x)} (as a \ww{\pi\sb{1}(X,x)}-module) gives rise to a functor
\w[,]{\Pi\sp{m}\sb{n}:\Pi\sb{1}X\to\cS} sending $x$ to the twisted Eilenberg-Mac~Lane
space \w[.]{K\sb{\pi\sb{1}(X,x)}(\pi\sb{n}(X,x),m)}
We denote \w{\hocolim\sb{\Pi\sb{1}X}\ \Pi\sp{m}\sb{n}(x)} by
\w[,]{K\sb{\Pi\sb{1}X}(\Pi\sb{n}X,m))} or simply \w[.]{K(\Pi\sb{n}X,m)}

By \cite[Proposition 2.4]{DKSmitO}, we have a homotopy pullback in spaces:
$$
\xymatrix@R=14pt@C=25pt{
\Po{n}X \ar[d] \ar[r] & \ar[d]\sp{i}N(\Pi\sb{1}X) \\
\Po{n-1}X \ar[r]\sp<<<<{k\sb{n-1}} & K(\Pi\sb{n}X, n+1)~,
}
$$
where $i$ is the zero section of the canonical projection map
\w{K(\Pi\sb{n}X,n)\to N(\Pi\sb{1}X)} to the nerve of \w[.]{\Pi\sb{1}X}

These constructions can be made natural in spaces $X$, so for each fibrant
simplicial category $\mX$ we have diagrams
$$
\xymatrix@R=14pt@C=25pt{
\Po{n}\mX \ar[d] \ar[r] & \ar[d]\sp{i}N(\Pi\sb{1}\mX)) \\
\Po{n-1}\mX \ar[r]\sp<<<<{k\sb{n-1}} & K(\Pi\sb{n}\mX, n+1)~,
}
$$
where \w{K(\Pi\sb{n}\mX,n+1)} and \w{N(\Pi\sb{1}\mX)} are the result of
applying the respective constructions to each mapping space.

By \cite[Proposition 3.2]{DKSmitO}, this diagram exhibits the $n$-th Postnikov section
as a small extension of the \wwb{n-1}st Postnikov section.
The map \w{k\sb{n}} is the called the \emph{$n$-th $k$-invariant} of $\mX$.

For \w[,]{n\geq 1} this may be obtained functorially by constructing the
homotopy pushout square
\mytdiag[\label{eqkinv}]{
  \Po{n}\mX
  \ar[d]\sb{p\sp{n}\sb{\mX}}\ar[r]\sp{i} & \pi\sb{1}\mX\ar[d]\sp{q}\\
\Po{n-1}\mX \ar[r]\sp{j} & \mZ~,
}
\noindent in \w[.]{\sCat} By \cite[Proposition 6.4]{BDGoeR}, we have
\w[,]{\Po{n+1}\mZ\simeq K\sb{\mD}(\Pi\sb{n}\mX,n+1)} (where \w[),]{\mD:=\Pi\sb{1}\mX}
and the composite \w{p\sp{n+1}\sb{\mZ}\circ q} represents
\w{k\sb{n-1}\in H\sp{n+1}(\Po{n-1}\mX;\pi\sb{n}\mX)}
(using the usual topological indexing \wh see Remark \ref{exam2.2}).

The Dwyer-Kan-Smith cohomology of a simplicial category can be viewed as a special
case of classical Quillen cohomology (see \cite[Proposition 2.4]{DKSmitO}), and thus
coincides with spectral Andre-Quillen cohomology by \S \ref{scaqc}.
Furthermore, the obstructions for solving lifting problems of the form:
\mytdiag[\label{eqpost}]{
\mA \ar[r] \ar[d] &     \Po{n}\mX \ar[d]  \\
\mB \ar@{.>}[ur] \ar[r] &   \Po{n-1}\mX
}
\noindent given by \cite[Proposition 2.4]{DKSmitO}) are a
special case of those of Theorem \ref{thm2.7}, by \cite[Proposition 2.2.3]{HNPrasA}.
\end{mysubsection}

\begin{remark}\label{rmk1.9}
We say that a fibrant simplicial category (quasi-category) is an \emph{$n$-category}
if its mapping spaces are \wwb{n-1}truncated (see \cite[Section 2.3.4]{LurieHTT}).
In this sense, \w{\Po{n}\mX} is an \wwb{n+1}category.

The mapping spaces in $\mX$ can be naturally identified with those in
\w{\fN(\mX)} (see \cite[Section 2.3.3]{LurieHTT}). Thus, \w{\fN (\Po{n}(\mX))}
is an \wwb{n+1}category and \w{\fN(\Po{n}(\mX))\to\fN(\mX)}
induces a bijection on homotopy groups of mapping spaces in degree $\le n$ and the zero map in
degrees \w[.]{> n} It follows from \cite[Proposition 2.3.4.12]{LurieHTT}
that we can identify \w{\fN (\Po{n}(\mX))} with the \wwb{n+1}truncation of the
quasi-category \w{\fN (\mX)} as defined in \cite[Section 2.3.4]{LurieHTT}).
\end{remark}

%
%
\sect{Cubical Sets and their Model Structure}
\label{ccsms}

In this section we recall the standard model structure on cubical sets
from \cite[Theorem 8.1]{JarCat}, and the model structures for
cubical quasi-categories from \cite{DKLSattC} (Quillen equivalent to the 
Joyal model structure on simplicial sets).
We show that each simplicially enriched category is canonically a homotopy
colimit of $n$-cubes.

\begin{mysubsection}{The category of cubes}
\label{sccube}
The \emph{cube category} $\square$ is the full subcategory of \w{\Cat} consisting of
objects of the form \w[,]{\bone{n}} where \w{\bon} is the ordered set \w{0<1} of
\S \ref{snac}. The category $\square$ is generated
by the following class of maps, where for \w{i=1,\cdots,n} and \w{\epsilon=0,1}
we have:
\begin{enumerate}
\renewcommand{\labelenumi}{(\alph{enumi})~}
\item \emph{Face maps} \w{d\sb{i, \epsilon}\sp{n} :\bone{n-1} \to\bone{n}}  given by
$$
d\sp{n}\sb{i, \epsilon}(x\sb{1}, \cdots, x\sb{n-1}) = (x\sb{1}, \cdots, x\sb{i-1},
\epsilon, x\sb{i} \cdots, x\sb{n-1})~;
$$
\item \emph{Degeneracy maps}:
\w{\sigma\sp{n}\sb{i} : \bone{n} \to \bone{n-1}} given by
$$
  \sigma\sp{n}\sb{i}(x\sb{1}, \cdots, x\sb{n}) = (x\sb{1}, \cdots, x\sb{i-1}, x\sb{i+1}
  \cdots, x\sb{n})~;
$$
\item  \emph{Positive connections}: \w{\gamma\sb{i, 0}\sp{n}:\bone{n}\to\bone{n-1}}
given by
$$
  \gamma\sb{i, 0}\sp{n}(x\sb{1}, \cdots, x\sb{n}) \mapsto (x\sb{1}, \cdots, x\sb{i-1},
  \max(x\sb{i}, x\sb{i+1}), \cdots x\sb{n})~;
$$
\item\emph{Negative connections}: \w{\gamma\sb{i, 1}\sp{n} : \bone{n} \to \bone{n-1}}
given by
$$
 \gamma\sb{i, 0}\sp{n}(x\sb{1}, \cdots, x\sb{n}) \mapsto (x\sb{1}, \cdots,
 x\sb{i-1}, \min(x\sb{i}, x\sb{i+1}), \cdots x\sb{n})~.
$$
\end{enumerate}

Write \w{\cSet} for \w[,]{\Set\sp{\sq{\op}}} and call the objects of this
category \emph{cubical sets}.
\end{mysubsection}

\begin{defn}\label{def3.1}
 Consider the functor \w{\square \to\sSet} given by
  \w[.]{\sq{n} \mapsto (\Delta\sp{1})\sp{n}} The left Kan extension of this functor
  along the Yoneda embedding defines a \emph{triangulation functor}, which we denote by
  \w[.]{T\sb{\bullet}} This fits into an adjunction:
$$
T\sb{\bullet} :\cSet \leftrightarrows\sSet : U\sb{\bullet}~.
$$
\end{defn}
\begin{defn}\label{cubicalnervedef}
There is a \emph{cubical nerve} functor \w[,]{N\sb{\square} :\Cat \to\cSet}
which takes $C$ to the cubical set \w[.]{\bone{n} \mapsto\hom(\bone{n}, C)}
We have a natural isomorphism
\w{N \simeq T\sb{\bullet}U\sb{\bullet}N \cong T\sb{\bullet}N\sb{\square}}
(the usual nerve \wh see \ref{def3.1} below for the definition \w[).]{U\sb{\bullet}}
Thus, we can identify \w{N\sb{\square}(\bone{n})} with \w[,]{N(\bone{n})}
and denote either by \w{\sq{n}} (in \w{\Set\sp{\Dop}} or \w[).]{\cSet}
\end{defn}

\begin{notn}
The \emph{boundary} \w{\partial \sq{n}} of \w{\sq{n}} is the largest proper subobject
of \w{\sq{n}} (a union of its proper faces).
The \emph{cubical horn} \w{\Pi\sb{i, \epsilon}\sp{n}} is the union of all proper
faces of \w{\sq{n}} except for the \wwb{i,\epsilon}-face.

The \emph{critical edge} of \w{\sq{n}} with respect to a face \w{d\sb{i, \epsilon}}
is the map \w{f:[1]\to[1]\sp{n}}
given by \w[,]{f\sb{i} =\id\sb{[1]}} and \w{f_j = \cons\sb{1- \epsilon}}
for \w[.]{j \neq i}

For \w[,]{n \ge 2} \w[,]{1 \le i \le n} and \w[,]{\epsilon = 0, 1}
the \wwb{(i, \epsilon)}\emph{open horn} \w{\widehat{\Pi}\sb{i,\epsilon}\sp{n}}
for the quotient of the horn \w{\Pi\sp{n}\sb{(i, \epsilon)}} by the critical edge of
\w[.]{d\sb{i, \epsilon}} The \wwb{(i,\epsilon)}\emph{open cube}
\w{\widehat{\square}\sp{n}\sb{i, \epsilon}} is the quotient of the cube by
the critical edge of \w[,]{d\sb{i, \epsilon}}
and the  \wwb{i,\epsilon}\emph{open horn inclusion} is the map
\w[.]{\widehat{\Pi}\sb{i,\epsilon}\sp{n} \to \widehat{\square}\sp{n}\sb{i, \epsilon}}
\end{notn}

\begin{defn}\label{def3.2}
A \emph{cubical Kan complex} is a cubical set $X$ such that it has the right
lifting property with respect to all horn inclusions
\w[.]{\Pi\sb{i, \epsilon}\sp{n} \to \sq{n}}.
An \emph{inner horn} is a horn \w{\Pi\sb{i, \epsilon}\sp{n}\to\sq{n}} with
\w[.]{0 < i < n}
A \emph{cubical quasi-category} is a cubical set $X$ such that \w{X \to\ast}
has the right lifting property with respect to the inner open horn inclusions.
\end{defn}

\begin{thm}[\protect{\cite[Section 8]{JarCat}}]\label{JardineThmCubical}
The category of cubical sets admits a \emph{Kan model structure}, in which the
cofibrations are monomorphisms and the fibrant objects are the cubical Kan complexes.
The adjunction of \S \ref{def3.1} induces an Quillen equivalence with the standard
model structure on simplicial sets.
\end{thm}

\begin{thm}\label{thm3.3} (\cite[Theorems 4.16, 4.2, \& 6.1]{DKLSattC})
The category of cubical sets admits a \emph{Joyal model structure},
in which the cofibrations are monomorphisms. and the fibrant objects are the
cubical quasi-categories, and the adjunction of \S \ref{def3.1} induces a Quillen
equivalence between this model structure and the Joyal model
structure on simplicial sets.
\end{thm}

\begin{remark}\label{rmk3.5}

In \cite{DKLSattC}, the authors defined two other variants of the Joyal model structure:

They first introduce a slight enlargement \w{\square\sb{\sharp}}
of $\square$, whose objects consist of \w{\bone{n}} for \w{n \in \NN} and
\w[.]{\bne} The maps of \w{\square\sb{\sharp}} are generated by those of $\square$
(\S \ref{sccube}), along with two maps \w{\phi:\bone{0}\to\bne} and
\w[,]{\zeta:\bne\to\bone{0}} subject to the identity
\w[.]{\zeta \phi = \sigma\sb{1}\sp{1}} We write
\w{\sq{1}\sb{e}} for the presheaf on \w{\square\sb{\sharp}} represented by \w[.]{\bne}

The objects of \w[,]{\cSet'':=\Set\sp{\square\sb{\sharp}\sp{\op}}}
are called \emph{structurally marked cubical sets}; those for which the map induced
by $\phi$ is a monomorphism are called simply \emph{marked cubical sets}, and the
full subcategory of such is denoted by \w[.]{\cSet'}

In the case of marked cubical sets, we can think
of the induced map \w{\phi : X\sb{e} \to X\sb{1}} as the inclusion of `marked edges'
of $X$. The \emph{minimal marking} is simply the inclusion of the degenerate edges;
the fact that the image of $\phi$ must contain the degenerate edges follows from the
identity \w[.]{\zeta \phi = \sigma\sb{1}\sp{1}}

By \cite[Proposition 4.3 and Theorem 3.7]{DKLSattC}, there are Quillen equivalences
$$
\cSet \leftrightarrows \cSet' \rightleftarrows\cSet''~.
$$
\end{remark}

%
%
\sect{Homotopy Colimit Decomposition into Cubes}
\label{chcdcs}

In this section we show how each cubical quasi-category, or fibrant simplicial category,
can be exhibited as a homotopy colimit of cubes:

\begin{thm}\label{thm4.1}
Any fibrant simplicial category $\mX$ is the homotopy colimit
(in the Bergner model structure) of the projection
\w{(\DK(\bon)\sp{\ast})\sb{/\mX} \to \sCat} (see \S \ref{snac}),
where \w{(\DK(\bon)\sp{\ast})\sb{/\mX}}
is \w[.]{\square \times\sb{\sCat} \sCat\sb{/\mX}}
\end{thm}

The proof of this Theorem requires several preliminary results.

\begin{lemma}\label{lem4.2}
The Joyal model structure on \w{\cSet''} from Remark \ref{rmk3.5}
is Quillen equivalent to a  Bousfield localization of the global
injective model structure on
\w[.]{(\Set\sp{\Dop})\sp{\square\sb{\sharp}}}
\end{lemma}

\begin{proof}
The construction in \cite[Theorem 3.6]{DKLSattC} implies that the model structure on
\w{\cSet''} can be obtained via the methods of Cisinski (\cite{CisiP}).
The class $\W$ of weak equivalences is thus an accessible localizer, in the
terminology of \cite{CisiP}. Using simplicial completion techniques
(i.e., \cite[Proposition 2.20]{AraHQ} or \cite[2.3.4.7]{CisiP}),
we can identify the model structure on \w{\cSet''} with a Cisinski model structure
on \w[,]{(\Set \sp{\Dop})\sp{\square\sb{\sharp}}} whose weak equivalences
we denote by \w[.]{\W\sb{\Delta}} It suffices to show that $\W$ is a
\emph{regular} localizer, in the terminology of \cite[Theorem 3.4.36]{CisiP}.
If the category \w{\square\sb{\sharp}} is \emph{squelettique}
(see \cite[Definition 8.1.1]{CisiP}), then all localizers are regular by
\cite[Proposition 8.2.2]{CisiP}.

By \cite[Remark 8.1.2]{CisiP}, it suffices to show that \w{\square\sb{\sharp}}
is a Reedy category in which \w{(\square\sb{\sharp})\sb{+}} and
\w[,]{(\square\sb{\sharp})\sb{-}} respectively, consists of monomorphisms and
 split epimorphisms. This  follows from \cite[Proposition 1.37]{DKLSattC}.
\end{proof}

Suppose that $X$ is a (quasi-)category, \w{C \subseteq X} a full subcategory, and
\w[.]{x \in X} We will write \w{C\sb{/x}} for the pullback:
$$
\xymatrix@R=14pt@C=25pt{
C\sb{/x} \ar[d] \ar[r] & X\sb{/x} \ar[d] \\
 C \ar[r] & X.
}
$$

The following is essentially \cite[Proposition 5.1.3.1]{LurieHTT}
(see also \cite[Lemma 2.4]{HaugC}):

\begin{lemma}\label{lem4.3}
  Let $J$ be a small quasi-category and \w{\phi : J\sp{\op} \to \cS}
  the presheaf corresponding to the right fibration \w[.]{E \to J}
  Then $\phi$ is the colimit of the composite:
$$
E \xra{p} J \xra{y} \cS\sp{J\sp{\op}}~,
$$
\noindent where $y$ is the Yoneda embedding.
\end{lemma}

\begin{lemma}\label{lem4.4}
  If \w{C\in\cSet''} is fibrant, it is a homotopy colimit of the projection map
  \w[.]{(\boldsymbol{\square}\sb{\sharp})\sb{/C} \to\cSet''}
\end{lemma}

\begin{proof}
The following argument is largely the same as that of \cite[Lemma 2.5]{HaugC},
which says that each quasi-category is a homotopy colimit of
its category of simplices:

By \cite[Proposition 2.20]{AraHQ}, we have a commutative diagram
$$
\xymatrix{
   \square\sb{\sharp} \ar[r]\sb{\phi|\sb{\square\sb{\sharp}}} \ar[d]\sb{y} &
   \square\sb{\sharp} \ar[d]\sb{y} \\
 \cSet'' \ar[r]\sb{\phi} & (\Set\sp{\Dop})\sp{\square\sb{\sharp}}~,
}
$$
\noindent where the bottom horizontal map is the Quillen
equivalence of Lemma \ref{lem4.2} and the vertical maps are
Yoneda. Furthermore the Quillen equivalence induces a Quillen
equivalence of slice model categories \w[.]{(\cSet'')\sb{/C}
\to ((\Set\sp{\Dop})\sp{\square\sb{\sharp}})\sb{/C}} Thus,
it suffices to prove the statement for the model structure
on \w{(\Set\sp{\Dop})\sp{\square\sb{\sharp}}} given by
Lemma \ref{lem4.2}, which is a left Bousfield localization of the global
injective model structure on \w[.]{(\Set \sp{\Dop})\sp{\square\sb{\sharp}}}
Using \cite[Theorem 4.2.4.1]{LurieHTT}, it suffices to show that the colimit of the
composite of the map of quasi-categories
\w{(\square\sb{\sharp})\sb{/C} \to \square\sb{\sharp} \xra{y}
\cS\sp{\square\sb{\sharp}}} is the functor represented by $C$. The
right fibration \w{(\square\sb{\sharp})\sb{/C} \to
\square\sb{\sharp}} classifies the functor in
\w{\cS\sp{\square\sb{\sharp}}} given by \w[:]{e
\mapsto\Map\sb{\square\sb{\sharp}}(e, C)} that is, the image of
$C$ under the Yoneda embedding. Thus, the result follows from
Lemma \ref{lem4.3}.
\end{proof}

\begin{lemma}\label{lem4.5}
Any minimally marked cubical quasi-category $C$ (\S \ref{rmk3.5})
is the homotopy colimit of the projection \w[.]{\square\sb{/C}\to\cSet''}
\end{lemma}

\begin{proof}
Since \w{\cSet''} is a left Bousfield localization of
\w[,]{(\Set \sp{\Dop})\sp{\square\sb{\sharp}}} by \cite[Theorem 4.2.4.1]{LurieHTT}
it suffices to prove that $C$ is the colimit of
\w{\square\sb{/C} \to \cS\sp{\square\sb{\sharp}}} in \w[.]{\cS\sp{\square\sb{\sharp}}}
By the dual \cite[Theorem 4.1.3.1]{LurieHTT} and the preceding paragraph,
it suffices to show that
\w{((\square\sb{\sharp})\sb{/C})\sb{\square\sb{e} \xra{\phi} C/}}
is weakly contractible for each \w[.]{\phi:\square\sb{e}\to C}

However, every map \w{\square\sb{e} \to C} factors through a map
\w[,]{\phi':\sq{0} \to C} since every marked edge of $C$ is degenerate, and
the diagram
$$
\xymatrix@R=15pt@C=25pt{
\square\sb{e} \ar[d] \ar[r]\sp{\phi} & C \\
\sq{0} \ar[ur]\sb{\phi'} &
 }
$$
\noindent is the terminal object in
\w[.]{(\square\sb{\sharp})\sb{/C})\sb{(\square\sb{e} \xra{\phi} C)/}}
\end{proof}

\begin{corollary}\label{cor4.6}
Every cubical quasi-category $X$ can be identified with the homotopy colimit
of \w[.]{\square\sb{/X} \to\cSet}
\end{corollary}

\begin{proof}
The composites of the left and right adjoint of the Quillen
equivalences of \cite[Theorem 3.7 and Proposition 4.3]{DKLSattC},
induce maps \w[,]{\cSet \to\cSet' \to\cSet''} so that the
composite takes each cubical set to a minimally marked cubical
set. The result follow from Lemmas \ref{lem4.4} and \ref{lem4.5}.
\end{proof}

\begin{corollary}\label{cor4.7}
Every quasi-category $X$ can be identified with the homotopy colimit
of \w[.]{\square\sb{/X} \to\sSet}
\end{corollary}

\begin{proof}
Using the fact that \w{T\sb{\bullet} \dashv U\sb{\bullet}} is a Quillen equivalence,
we have a zig-zag of weak equivalences:
$$
X~\xleftarrow{\simeq}~T\sb{\bullet}U\sb{\bullet}X~\xra{\simeq}~
T\sb{\bullet} \underset{\sq{n}\to U\sb{\bullet}X}{\hocolim}~\sq{n}~\xra{\simeq}~
\underset{\sq{n} \to U\sb{\bullet}X}{\hocolim} T\sb{\bullet}~\sq{n}~=~
\underset{T\sb{\bullet}\sq{n} \to X}{\hocolim} T\sb{\bullet}~\sq{n}~.
$$
\end{proof}

\begin{proof}[Proof of Theorem \ref{thm4.1}]
Let $\mX$ be a fibrant simplicial category. Using Corollary \ref{cor4.7} and
  \cite[Theorem 6.7]{Riehl1}, we have the following weak equivalences
$$
\mX \simeq \fC(\underset{\sq{n} \to \fN(\mX) \in S\sb{\mX}}{\hocolim} \sq{n}) \simeq
\underset{\DK(\bone{n}) \to \mX}{\hocolim} \fC(\sq{n})
\simeq \underset{\DK(\bone{n}) \to \mX}{\hocolim} \DK(\bone{n})~.
$$
\end{proof}

\begin{defn}\label{def4.8}
A simplicial or cubical set is a \emph{finite cubical complex} if it is
a finite iterated pushout of cubes along face maps.
\end{defn}

We call an cubical set a \emph{finite cubical complex} if it can be expressed
as a finite pushout of cubes along face maps.

\begin{lemma}\label{lem4.9}
  Any simplicial set $X$ can be identified with the homotopy colimit of its
  finite cubical subcomplexes.
\end{lemma}

\begin{proof}
Given a finite cubical complex $K$, we will write \w{\square\sb{/K}\sp{\inc}} for
the category consisting of face inclusions in $K$ and face maps between them.

Let $J$ be the poset of finite cubical subcomplexes of $X$. and let
  \w{\phi : J \to \Set \sp{\Dop}} be the functor
\w[.]{K \mapsto \square\sb{/K}\sp{\inc}},
where the superscript in means inclusions of faces in $K$
Since \w[,]{\square\sb{/X} = \cup\sb{K \in J} \square\sb{/K}\sp{\inc}} it follows from
\cite[Remark 4.2.3.9]{LurieHTT} and Theorem \ref{thm4.1} that
$$
X = \underset{K \in J}{\hocolim} \,\square\sb{/K}\sp{\inc}  =
\underset{K \in J}{\hocolim} \,K
$$
\end{proof}

%
%
\sect{Colimit Decomposition into Cubical Complexes}
\label{chcdcscc}

In this section, we show that in certain situations the homotopy colimit
decomposition of the preceding section simplifies to a strict colimit. This yields
an even simpler description of the colimit decomposition of cohomology.

In particular, a cubical complex can be written as a strict
colimit of face maps between cubes, which is also a homotopy
colimit. However, this is not always useful from our perspective,
so we prove an additional colimit decomposition statement (Theorem
\ref{thm5.8}), which applies in greater generality.

\begin{defn}\label{def5.1}
Let $X$ (respectively, $\mX$) be a cubical set (respectively, category).
A  map \w{\sq{n} \to X}
(respectively, \w[)]{\DK(\bone{n}) \to \mX} is called
\emph{degenerate} if it factors through either a degeneracy or a
connection. We will write \w{\square\sb{/X}\sp{\ND}} (respectively,
\w[)]{(\DK\bone{\ast})\sb{/\mX}\sp{\ND}} for the category whose object are non-degenerate
cubes over $X$ (respectively, $\mX$) and whose morphisms are face maps between
non-degenerate cubes.
\end{defn}

\begin{defn}\label{def5.2}
We say that a simplicial or cubical set is a simplicial or cubical
\emph{complex} if each face of a non-degenerate simplex (cube) is non-degenerate
\end{defn}

By \cite[Proposition 5.2.1]{Fjellbo}, every simplicial set $X$ is a colimit of
its non-degenerate simplices under the face maps, which can be expressed
as an iterated pushout of face maps, and is thus in fact a homotopy colimit
This result can be adapted to yield:

\begin{prop}\label{thm5.3}
Every cubical complex $X$ can be identified with the homotopy colimit
of the projection map \w[.]{(\square\sb{/X}\sp{\ND}) \to \cSet}
\end{prop}

\begin{proof}
By \cite[Proposition 1.18]{DKLSattC}, each map \w{\sq{n}\to X} can be factored
as a degenerate map followed by a non-degenerate one, and as in the proof
of \cite[\emph{loc.\ cit.}]{Fjellbo} we deduce that $X$ can be written
as the (homotopy) colimit of \w[.]{\square\sb{/X}\sp{\ND} \to \cSet}
\end{proof}

\begin{remark}\label{rmk5.4}
  The above colimit decomposition is not sufficient for all of our purposes.
  Barycentric and cubical subdivision (for the latter, see \cite{JarCub})
  destroy fibrancy and do not preserve Joyal homotopy type.
  Thus one cannot apply the decomposition of Proposition \ref{thm5.3} directly
  to give an appropriate homotopy colimit decomposition of a general
  (cubical) quasi-category. We will prove an appropriate statement in
  Theorem \ref{thm5.8} below.
\end{remark}

\begin{lemma}\label{lem5.5}
Suppose that $\mX$ is a fibrant simplicial category, which is cofibrant in the
Bergner model structure. Then each composition operation in $\mX$ is a monomorphism.
\end{lemma}

\begin{proof}
  This follows from the characterization of \cite[Theorem 7.6]{DKanL},
  which says that each cofibrant simplicial category is a retract of one which is
  free in each simplicial degree.
\end{proof}

\begin{lemma}\label{lem5.6}
Suppose that $\mX$ is a fibrant cubical category in which each composition operation
is monic. Suppose that \w{\phi: \sq{n} \to \fN(\mX)}
is an $n$-cube. Let $0$ and $t$ denote the initial and terminal objects of
\w[,]{\sq{n}} respectively. Then $\phi$ is degenerate if and only if the map
\w{\fC(\square \sp{n})(0, t) \xrightarrow{\phi} \mX(\phi(0), \phi(t))}
induced by the adjoint map \w{\phi: \fC (\sq{n}) \to \mX} factors through some
$$
    (\square \sp{n})(0, t) \xrightarrow {\gamma\sb{i, \epsilon}}
    (\square \sp{n-1})(\gamma\sb{i, \epsilon}(0), \gamma\sb{i, \epsilon}(t)),\text{ or}
    (\square \sp{n})(0, t)
    \xrightarrow{\sigma\sb{i}} (\square \sp{n-1})(\sigma\sb{i}(0), \sigma\sb{i}(t))~.
$$
\end{lemma}

\begin{proof}
Sufficiency is clear. For necessity, we can assume \w[,]{n \ge 2} since the
statement is trivially true for \w[.]{n = 1}

Without loss of generality, assume that
\w{\Map \sb{\fC(\sq{n})}(0,t) \xrightarrow{\phi}\Map\sb{\mX}(\phi(0),\phi(t))}
factors through a degeneracy \w[.]{\sigma\sb{a}}

Choose $i$ and $j$ in \w[,]{\{ 0, \cdots n\}} and let \w{e\sb{1} : 0 \to j}
and \w{e\sb{2} : j \to n} be morphisms in
\w[.]{\fC  (\sq{n})} We have a diagram
$$
\xymatrix
{
  \Map\sb{\fC (\sq{n})}(i, j) \ar@/_6pc/[dd]\sb{\phi\sb{i, j}}
  \ar[d]\sb{\fC ( \sigma\sb{a})} \ar[rr]\sb{(e\sb{1*}, e\sb{2}\sp{\ast})} &&
  \Map\sb{\fC(\sq{n})}(0, t) \ar[d] \\
  \Map \sb{\fC(\sq{n-1})} (\sigma\sb{a}(i), \sigma\sb{a}(j))
  \ar[rr] \ar@{.>}[d]\sb{\phi'(i, j)}  \ar[rr]\sb{(\sigma \sb{a}
    (e\sb{1\ast}), \sigma \sb{a} (e\sb{2}\sp{\ast}))} &&
  \Map \sb{\fC  (\sq{n-1})}(\sigma\sb{a}(0), \sigma\sb{a}(t))
  \ar[d]\sp{\phi'(0, t)}   \\
  \Map\sb{\mX}(\phi(i), \phi(j))
  \ar[rr]\sb{(\phi(e\sb{1})\sb{\ast},\phi(e\sb{2})\sp{\ast})} &&
  \Map\sb{\mX}(\phi(0), \phi(t))~.
}
$$

Since composition is monic in $\mX$,
\w{(\phi(e\sb{1})\sb{\ast}, \phi(e\sb{2})\sp{\ast})} is a monomorphism, too,
and we can define \w{\phi'(i, j)} as
$$
(\phi(e\sb{1})\sb{\ast}, \phi(e\sb{2})\sp{\ast})\sp{-1} \circ \phi'(0, t)
\circ (\sigma \sb{a} (e\sb{1\ast}), \sigma \sb{a} (e\sb{2}\sp{\ast}))~.
$$

This definition is independent of the choices of \w{e\sb{1}} and
\w[,]{e\sb{2}} since the composite \w[,]{\phi'(i, j) \circ \fC  (\sigma \sb{a})
= \phi\sb{i, j}} \w{\phi\sb{i,j }} is independent of the choice of
\w{e\sb{1} : 0 \to j} and \w[,]{e\sb{2} : j \to n}
and \w{\fC (\sigma\sb{a})} is surjective on mapping spaces (since
it has a section induced by a face map).

We now show that \w{\phi' : \fC(\sq{n-1}) \to \fC(\mX)} is
a natural transformation. It suffices to show commutativity of the bottom square in:
$$
\xymatrix
{
  \Map \sb{\fC   (\sq{n})}(i, j)  \times \Map \sb{\fC  (\sq{n})}(j, k)
  \ar[r]_>>>>>>>>>>>{c} \ar[d]\sb{\sigma \sb a \times \sigma \sb a} &
  \Map \sb{\fC   (\sq{n})}(i, k) \ar[d]\\
  \Map \sb{\fC  (\sq{n-1})}(\sigma \sb{a}(i), \sigma \sb{a}(j))  \times
  \Map \sb{\fC  (\sq{n-1})}( \sigma \sb{a}(j), \sigma \sb{a}(k))
  \ar[r]_>>>>>{c} \ar[d]\sb{\phi' \times \phi'} &
  \Map \sb{\fC  (\sq{n-1})}(\sigma \sb{a}(i),\sigma\sb{a}(k)) \ar[d]\sp{\phi'}  \\
  \Map \sb{\mX }(\phi(i), \phi(j)) \times \Map \sb{\mX }(\phi(j), \phi(k))
  \ar[r]_>>>>>>>>{c} &   \Map \sb{\mX }(\phi(0), \phi(t))~.
}
 $$
\noindent Here $c$ denotes composition.

The commutativity of the outer square and the top square imply that we have:
$$
c \circ (\phi' \times \phi') \circ (\sigma \sb{a} \times \sigma \sb{a}) =
\phi' \circ c \circ (\sigma \sb{a} \times \sigma \sb{a})~,
$$
so since \w{\sigma \sb{a} \times \sigma \sb{a}} is surjective, we have
$$
c \circ (\phi' \times \phi')   = \phi' \circ c
$$
\end{proof}

\begin{defn}
A \emph{weakly initial} object in a category $C$ is an object $x$ such that
\w{\hom(x, y)} is nonempty for each \w[.]{y \in C}
\end{defn}

\begin{lemma}\label{lem5.7}
Let \w{x : \sq{n} \to X} be a cube in a quasi-category $X$. Then
$$
\square\sb{/X}\sp{\ND} \times\sb{\square\sb{/X}} (\square\sb{/X})\sb{\square\sp{n}
  \xrightarrow{x} X /}
$$
has a weakly initial object.
\end{lemma}

\begin{proof}
By \cite[Proposition 1.18]{DKLSattC}, we can write $x$ as
\w[,]{\sq{n} \xrightarrow{z} \sq{n-k} \xrightarrow{y} X} where $y$
is non-degenerate and $x$ is a composite of degeneracies and connections.
We claim that $y$ is weakly initial in the above category.
Suppose that we have a commutative diagram
$$
 \xymatrix
 {
\sq{n} \ar[r]\sb{z} \ar[d]\sb{r} & \sq{n-k} \ar[d]\sp{y} \\
\sq{m} \ar[r]\sb{w} & \mX`,
 }
$$
where $w$ is non-degenerate. We can use the cubical identities to show that
$z$ has a section $q$ consisting of face maps, so we also have a commutative diagram:
$$
\xymatrix
{
   \sq{n-k} \ar[r]\sb{q} \ar[dr] \ar@/^2ex/[rr]\sp{id} &
   \ar[d]\sb{r} \sq{n} \ar[r]\sb{z} & \sq{n-k} \ar[d]\sp{y} \\
& \sq{m} \ar[r]\sb{w} & \mX
}
$$
\noindent The composite \w{r \circ q} must be a composite of face maps by
 \cite[Proposition 1.18]{DKLSattC}, since $y$ is non-degenerate.
 Thus, we have shown this object is weakly initial.
\end{proof}

\begin{thm}\label{thm5.8}
If $\mX$ is a fibrant simplicial category in which each composition
map is a monomorphism, then \w[.]{\fN(\mX)=X=\underset{(\sq{n} \to X)
\in\sq{\ND}\sb{/X}}{\hocolim} \sq{n}}
\end{thm}

\begin{proof}
By Theorem \ref{thm4.1}, it suffices to show that the inclusion functor
\w{\sq{\ND}\sb{/X}) \to (\sq{\ast}\sb{/X})} is cofinal.
By \cite[Proposition 4.1.3.1]{LurieHTT}, it suffices to show that
$$
(\sq{\ND}\sb{/X}) \times\sb{(\sq{\ast}\sb{/X})}
(\sq{\ast}\sb{/X})\sb{\square \xrightarrow{x} X/}
$$
\noindent is weakly contractible. We claim that the object described in the
preceding lemma is in fact initial.

Suppose by way of contradiction that we have a commutative diagram
$$
\xymatrix
{
  \sq{i} \ar[dr]\sb{z} \ar@/_1ex/[r]\sb{\phi} \ar@/^1ex/[r]\sp{\phi'} &
  \ar[d]\sp{y} \sq{m}  \\
 & X~,
}
$$
where $z$ is non-degenerate and \w{\phi \neq \phi'} are composites of face maps.
Let \w{0\sb{i}} and \w{t\sb{i}} denote the initial and terminal objects of
\w[,]{\fC(\sq{i})} respectively. Consider the adjoint diagram:
$$
\xymatrix
{
  \fC  (\sq{i}) \ar@/_1ex/[r]\sb{\phi} \ar@/^1ex/[r]\sp{\phi'} \ar[dr] &
  \ar[d]\sp{y}  \fC (\sq{n}) \\
& \mX
}
$$
which induces
$$
\xymatrix
{
  \fC(\sq{i})(0\sb{i}, t\sb{i}) \ar@/_1ex/[r]\sb{\phi} \ar@/^1ex/[r]\sp{\phi'}
  \ar[dr]\sb{z} & \ar[d]\sp{y} \fC (\sq{n})(\phi(0\sb{i}), \phi(t\sb{i})) \\
& \mX(y(0\sb{i}), y(t\sb{i}))~.
}
$$
Then \w[,]{\Image (y \circ \phi) = \Image (y \circ \phi') = \Image(z)}
so \w[.]{\Image(z) = y (\Image(\phi) \cap \Image(\phi'))}
We can identify $\phi$ and \w{\phi'} with $i$-faces of \w{\sq{n}}
whose intersection can be identified with an \wwb{i-m}face
\w[.]{\phi'' : \sq{i-m} \to X}
Therefore, \w{y (\Image(\phi) \cap \Image(\phi'))} can be identified with the image of
$$
\fC(\sq{i-m})(0\sb{i-m}, t\sb{i-m}) \to
\fC(\sq{n})(\phi''(0\sb{i-m}), \phi''(t\sb{i-m})) \xrightarrow{y}
\mX(y \circ \phi''(0\sb{i-m}), y \circ \phi''(t\sb{i-m}))~,
$$
so that $z$ is degenerate by Lemma \ref{lem5.7}, contrary to the hypothesis.
\end{proof}

\begin{corollary}\label{cor5.9}
  Suppose that $\mX$ is a fibrant simplicial category in which each composition
  operation is a monomorphism. Then $\mX$ can be identified with the homotopy
  colimit of the projection map \w[.]{(\DK\bone{\ast})\sb{/\mX}\sp{\ND} \to \sCat}
\end{corollary}

%
%
\sect{Spectral Andr\'{e}-Quillen cohomology of cubes}
\label{csaqcc}

In this section we describe a cochain complex calculating the spectral
Andr\'{e}-Quillen cohomology of an $n$-cube. The homotopy colimit decompositions
described above then yield (more or less explicit) cochain complexes for more
general simplicial categories, allowing one to reduce the obstructions of
Proposition \ref{thm2.4} to individual cubes.

\begin{defn}\label{def6.1}
If $P$ is a poset, the set \w{[x, y] := \{z : x \le z \le y \}} will be called
an \emph{interval} in $P$.
\end{defn}

\begin{lemma}\label{lem6.2}
  Suppose that $P$ is a poset. Then \w{\Tw(P)} is equivalent to the poset of
  intervals of $P$, ordered by inclusion.
\end{lemma}

\begin{proof}
We define a functor which takes each \w{f : x \to y} to \w[,]{[x, y]}
and the unique morphism between $f$ and $g$ (if it exists):
$$
\xymatrix@R=18pt@C=25pt{
a \ar[d]\sb{f}  & c \ar[l]  \ar[d]\sp{g} \\
b \ar[r] & d
}
$$
to the interval inclusion \w[.]{[a, b] \subseteq [c, d]} This is evidently a
bijection on morphisms and objects.
\end{proof}

Suppose that $P$ is a poset with \w{x, y \in P} such that every maximal chain
\w{x < a\sb{1} < \cdots a\sb{n-1} < y} has the same number of elements $n$. We then say
that the interval \w{[x, y]} has \emph{length} $n$ and write \w[.]{\ell(x, y ) = n}
If there are no morphisms \w{x \to y} we say that \w{[x, y]} has length 0.

As an example, if \w{x=(x\sb{1}, \cdots, x\sb{n})} and \w{y=(y\sb{1}, \cdots, y\sb{n})}
in \w[,]{\bone{n}} then
\begin{equation*}
\ell(x, y) = \begin{cases}
  (y\sb{1} - x\sb{1}) + \cdots (y\sb{n} - x\sb{n}) &
  \text{if }  x\sb{1} \le y\sb{1}, \cdots, x\sb{n} \le y\sb{n}\\
               0 & \text{otherwise}
          \end{cases}
\end{equation*}

\begin{construction}\label{con6.3}

Suppose that $P$ is a poset of \emph{bounded interval length} (that is, there is an
\w{n\in\NN} such that each interval has length \w[),]{\leq n} and
\w{\phi:\Tw(P) \to \AbGp} is a functor.
Then we can form a complex \w{\fD\sb{\phi}} as follows:
\begin{myeq}[\label{complex}]
  \bigoplus\sb{l[a\sb{0}, b\sb{0}] = 0} \phi(a\sb{0}, b\sb{0})
  \xra{\partial\sp{0}} \bigoplus\sb{l[a\sb{1}, b\sb{1}] = 1}
  \phi(a\sb{1}, b\sb{1}) \to \cdots \xra{\partial\sp{n-1}}
  \bigoplus\sb{l[a\sb{n}, b\sb{n}] = n} \phi(a\sb{n}, b\sb{n})~.
\end{myeq}
\noindent (concentrated in degrees \w{-n+1} to $1$).

For any \w[,]{a \le b} set
\begin{equation*}
\begin{split}
(\phi\sb{a}\sp{b})\sp{\ast}~:=&~\phi([d, a] \subseteq [d, b]):\phi(d, a) \to \phi(d, b)
  \hs\text{and}\\
(\phi\sb{a}\sp{b})\sb{\ast}~:=&~\phi([b, c] \subseteq [a, c]):\phi(b, c) \to \phi(a, c)~.
\end{split}
\end{equation*}
\noindent Note that
$$
(\phi\sb{b}\sp{c})\sp{\ast}\circ(\phi\sb{a}\sp{b})\sp{\ast}~=~(\phi\sb{a}\sp{c})\sp{\ast},
(\phi\sb{a}\sp{b})\sb{\ast}\circ(\phi\sb{b}\sp{c})\sb{\ast}~=~
(\phi\sb{a}\sp{c})\sb{\ast}~.
$$

Suppose \w{[a\sb{i+i}, b\sb{i+1}]} satisfies
\w[.]{\ell([a\sb{i+i}, b\sb{i+1}]) = i+1} The value of the coboundary on the component
\w[,]{[a\sb{i+i}, b\sb{i+1}]} which we denote by
\w[,]{\partial\sp{i-1}\sb{[a\sb{i+i}, b\sb{i+1}]}} is given by:
\begin{equation*}
  \partial\sp{i-1}(x)\sb{[a\sb{i+1}, b\sb{i+1}]}
  = \begin{cases}
    (\phi\sb{b\sb{i}}\sp{b\sb{i+1}})\sp{\ast}  &
    \text{if }~a\sb{i+1} = a\sb{i} \\
    (-1)(\phi\sb{a\sb{i}}\sp{a\sb{i+1}})\sb{\ast} &
    \text{if}~ b\sb{i+1} = b\sb{i} \\
    0 & \text{otherwise.}
\end{cases}
\end{equation*}

It is easy to check that this is in fact a complex. We will show without loss of
generality that \w[.]{\partial \sp{2i} \circ \partial \sp{2i-1} = 0}

Suppose that \w{[a, b]} is an interval of length \w[,]{2i+1}
and let $S$ be the set of intervals \w{[a\sb{s}, b\sb{s}] \subseteq [a, b]}
with length \w{2n-1} and \w[.]{a\sb{s} \neq a, b\sb{s} \neq b}

The value of \w{\partial \sp{2i} \circ \partial \sp{2i-1}} on the component \w{[a, b]}
is zero on intervals not in $S$, and is given by the composite
$$
\phi(a\sb{s}, b\sb{s})
\xrightarrow{((\phi\sb{b\sb{s}}\sp{b})\sp{\ast}, -\phi\sb{a}\sp{a\sb{s}})}
\phi(a\sb{s}, b) \oplus \phi(a, b\sb{s}) \xrightarrow{(\phi\sb{a\sb{s}}\sp{a})\sb{\ast},
  (\phi\sb{b\sb{s}}\sp{b})\sp{\ast}} \phi(a, b)
$$
on intervals in $S$, which is in fact also equal to
\begin{equation*}
\begin{split}
(\phi\sb{a\sb{s}}\sp{a})\sb{\ast} \circ (\phi\sb{b\sb{s}}\sp{b})\sp{\ast} +
(\phi\sb{b\sb{s}}\sp{b})\sp{\ast} \circ (-\phi\sb{a\sb{s}}\sp{a})\sb{\ast} ~&=~
(\phi\sb{a\sb{s}}\sp{a})\sb{\ast} \circ (\phi\sb{b\sb{s}}\sp{b})\sp{\ast} -
  (\phi\sb{b\sb{s}}\sp{b})\sp{\ast} \circ (\phi\sb{a\sb{s}}\sp{a})\sb{\ast}\\
 & =~
(\phi\sb{a\sb{s}}\sp{a})\sb{\ast} \circ (\phi\sb{b\sb{s}}\sp{b})\sp{\ast}
- (\phi\sb{a\sb{s}}\sp{a})\sb{\ast} \circ (\phi\sb{b\sb{s}}\sp{b})\sp{\ast}~=~0~.
\end{split}
\end{equation*}
\end{construction}

\begin{lemma}\label{lem6.4}
Suppose that $\mX$ is a fibrant simplicial category, $C$ is an ordinary category,
and we have a morphism \w[.]{\phi : C \to \mX}
Suppose that \w{M \in \mT\sb{\mX}(\Cat_\infty)} is the small extension considered in
\S \ref{skinv}, so that
$$
\Omega^\infty(M) = \mX \to K(\Pi\sb{n}(\mX), n) \to  \mX.
$$
Under the equivalence of Theorem \ref{thm2.6}, the object \w{\phi\sp{\ast}(M)}
corresponds to the nerve of a functor
\w{\F : \Tw(C) \to \Sp(\sSet)} given on objects by
$$
(f: x\to y)~\mapsto~K(\pi\sb{n}(\Map\sb{\mX}(x, y), f ), n)
$$
and on morphisms
$$
\xymatrix@R=15pt@C=25pt{
x \ar[d]_h  & \ar[d]\sb{i} \ar[l]_g z \\
y \ar[r]\sb{f} & w
}
$$
by the map of Eilenberg-Mac~Lane spaces induced by (pre)composition
$$
\pi\sb{n}(\Map\sb{\mX}(x, y), h ) \xrightarrow{g\sb{\ast}} \pi\sb{n}(\Map\sb{\mX}(z, y), h \circ g )
\xrightarrow{f\sp{\ast}} \pi\sb{n}(\Map\sb{\mX}(z, w), i )~.
$$
\end{lemma}

\begin{proof}
To simplify the notation, let us write
$$
A\sb{f} := \pi\sb{n}(\Map\sb{\mX}(\phi(x), \phi(y)), \phi(f))
$$
for each \w[.]{(f : x \to y) \in \Mor(C)}

We can identify \w{\phi\sp{\ast}(M)} with a diagram
$$
C \to \mZ := C \times\sb{\mX} K(\Pi\sb{n}( \mX ), n) \xrightarrow{\proj} C  ~.
$$

We have \w[.]{\Obj(\mZ) = \Obj(C)}
Since pullbacks of simplicial categories induce pullbacks of mapping spaces,
we also have:
$$
\Map\sb{\mZ}(x, y) = \coprod\sb{f \in \Map\sb{C}(x, y)} K(A\sb{f}, n)
$$

We will now identify \w{\phi\sp{\ast}(M)} as a functor \w[.]{\F: N\Tw(C) \to \Sp}
  By \cite[Theorem 3.1.14]{HNPrasA}, we can identify \w{\phi\sp{\ast}(M)}
    with the image under stabilization of the object $\G$
of \w{\Fun(N(C) \times (N(C))\sp{op}, \cS)\sb{\Map\sb{N(C)} // \Map\sb{N(C)}}}
given by
$$
(x, y) \mapsto \Map\sb{\mZ}(f(x), f(y)).
$$

Let \w{\gamma : N\Tw(C) \to N(C) \times N(C)\sp{\op}} be the fibration which classifies
the functor \w{N(C) \times N(C)\sp{\op} \to \cS} given by
\w[.]{(x, y) \mapsto \Map_C(x, y)}
By   \cite[Section 3.3]{HNPrasA}, under unstraightening, the objects $\G$ and $\F$
correspond to the maps $\eta$ and $\psi$, respectively,  in a diagram of left fibrations
\begin{equation}\label{classifyingfibs}
\xymatrix
{
  N\Tw(C) \ar[drr] \ar[r] & X \ar[dr]^>>>>>>>>{\eta} \ar[r]\sp{\psi} &
  N\Tw(C) \ar[d]\sb{\gamma} \\
& & N(C) \times N(C)\sp{\op}~,
}
\end{equation}
and the map $\psi$ corresponds to the natural transformation of functors
$$
\Map\sb{\mZ}(\phi(-),\phi(-)) \to \Map\sb{C}((-), (-))
$$
induced by \w[.]{\proj} Thus, we see that $\F$ acts on objects by
$$
f \mapsto K(A\sb{f}, n)~.
$$

Let \w{(f, g) : h \to i} be a morphism of \w{\Tw(C)} depicted by
$$
\xymatrix@R=15pt@C=25pt{
x \ar[d]_h  & \ar[d]\sb{i} \ar[l]_g z \\
y \ar[r]\sb{f} & w
}
$$

Let \w{X\sb{(f, g)}} (respectively, \w[)]{X\sb{\gamma(f, g)}} denote
the pullback of \w{(f, g)} (respectively, \w[)]{\gamma(f, g)} along $\psi$
(respectively, $\eta$). Then we have an induced map of left fibrations
$$
\xymatrix@R=15pt@C=25pt{
X\sb{(f, g)}  \ar[dr] \ar[r]  & X\sb{\gamma(f, g)} \ar[d] \\
 & [1]
}
$$
which under straightening corresponds to a natural transformation which we can depict as
$$
\xymatrix@R=15pt@C=25pt{
 K(A_h, n) \ar[d] \ar[r]_d  & \ar[d] K(A\sb{i}, n) \\
 \coprod\sb{q \in \Map_C(x, y)}  K(A_q, n) \ar[r]_c & \coprod\sb{r \in \Map_C(z, w)}
 K(A_r, n)
}
$$
with the vertical maps inclusions.

By the definition of $\eta$, the map $d$ is given by
$$
\Map\sb{\mZ}(f(x), f(y)) \xrightarrow{g\sb{\ast}}
\Map\sb{\mZ}(f(z), f(y))\xrightarrow{f\sp{\ast}} \Map\sb{\mZ}(f(z), f(w))~,
$$
which, when restricted to \w[,]{K(A_h, n)}
is the map of $n$-dimensional Eilenberg-Mac~Lane spaces corresponding to
$$
\pi\sb{n}(\Map\sb{\mX}(x, y), h ) \xrightarrow{g\sb{\ast}}
\pi\sb{n}(\Map\sb{\mX}(z, y), h \circ g ) \xrightarrow{f\sp{\ast}}
\pi\sb{n}(\Map\sb{\mX}(z, w), i )~.
$$
\end{proof}

\begin{prop}\label{thm6.5}
Given a diagram
$$
\xymatrix@R=18pt@C=25pt{
& \Po{n}\mX \ar[d] \\
[1]\sp{m} \ar[r]\sb{\phi} \ar@{.>}[ur] & \Po{n-1}\mX
}
$$
\noindent where $\mX$ is a fibrant simplicial category, consider the
cochain complex \w{C\sb{\phi}}
\begin{multline*}
  \bigoplus\sb{\ell([x\sb{0}, y\sb{0}]) = 0}
  \pi\sb{n}(\Map\sb{\mX}(x\sb{0}, y\sb{0}), \phi(x\sb{0} \to y\sb{0})) \to
  \bigoplus\sb{\ell([x\sb{1}, y\sb{1}]) = 1}
  \pi\sb{n}(\Map\sb{\mX}(x\sb{1}, y\sb{1}) , \phi(x\sb{1} \to y\sb{1})) \\
  \cdots \to \bigoplus\sb{l[x\sb{m}, y\sb{m}] = m}
  \pi\sb{n}(\Map\sb{\mX}(x\sb{m}, y\sb{m}), \phi(x\sb{m} \to y\sb{m}))
\end{multline*}
\noindent concentrated in degrees \w[,]{[-n, m-n]} where the intervals range over
all intervals of the cube, with coboundaries as in \S \ref{con6.3}. Then we can
compute \w{H\sp{\ast}([1]\sp{n+1}, \pi\sb{n}(\mX)} using \w[.]{C\sb{\phi}} In particular,
there is a canonical class \w{\beta \in H\sp{1}(C\sb{\phi})} whose vanishing determines
whether the lift exists.
\end{prop}

\begin{proof}
Recall from Lemma \ref{lem6.2} that we can identify \w{\Tw([1]\sp{m})} with the poset
of interval inclusions in \w[.]{[1]\sp{m}}
Consider the diagram \w{\F : \Tw([1]\sp{m}) \to \AbGp}
to abelian groups given by
$$
[a, b] \mapsto \pi\sb{n}(\Map\sb{\mX}(\phi(a), \phi(b)), \phi(a \le b))~,
$$
which takes \w{([a, b] \subseteq [c, d])} to
\begin{equation*}
\begin{split}
  \pi\sb{n}(\Map\sb{\mX}(\phi(a), \phi(b)), \phi(a \to b) ) &
  \xrightarrow{(c \to a)\sb{\ast}} \pi\sb{n}
  (\Map\sb{\mX}(\phi(c), \phi(a),\phi(c \to b) ) \\
  & \xrightarrow{(b \to d)\sp{\ast}} \pi\sb{n}(\Map\sb{\mX}(\phi(a), \phi(d)),
  \phi(c \to d) )~.
\end{split}
\end{equation*}
\noindent It thus follows that under the identification of Theorem \ref{thm2.6},
\w{\pi\sb{n}(\mX)} corresponds to diagram of Eilenberg-Mac~Lane spaces
given by \w[.]{K(-, n) \circ \F} By \cite[Corollary 3.3.2]{HNPrasA}, this means that
\w{H\sp{-k}(P, \pi\sb{n}(\mX))} corresponds to the $k$-th total derived functor
of the homotopy limit of the diagram \w[.]{\F[n]}
But this homotopy limit can be identified with the  homotopy limit of the complex
\w{\fD_\phi[n]} from \wref[.]{complex}
The $k$-th total derived functor of this complex can be identified with its
\wwb{-k}th cohomology group, hence the result.
\end{proof}

We now show how spectral Andr\'{e}-Quillen cohomology can be decomposed in
terms of the spectral Andr\'{e}-Quillen cohomology of cubes, yielding an explicit chain
complex model.

\begin{defn}\label{dcochaincx}
For any simplicial category $\mX$ and map
\w{\phi:\DK([1]\sp{k}) \to \mX} let \w{C\sb{\phi}}
be the cochain complex defined in Proposition \ref{thm6.5}. For each
  \w{1\leq i\leq k} and \w{\epsilon\in\{0,1\}} the map
  \w{d\sp{i, \epsilon}:[1]\sp{k}\to[1]\sp{k+1}} induces a projection
  \w{(d\sp{i, \epsilon})\sp{\ast}} from \w{C\sb{\phi}} onto the subcomplex
\w{C\sb{\phi \circ d\sp{i, \epsilon}}} of \w[.]{C\sb{\phi}}
We denote by \w{C\sp{\ast}(\mX)} the limit of the
diagram of cochain complexes given by the maps
\w{(d\sp{i, \epsilon})\sp{\ast}:C\sb{\phi}\to C\sb{\phi \circ d\sp{i, \epsilon}}}
for all non-degenerate maps \w[,]{\phi : \DK([1]\sp{k+1}) \to \mX}
\w{1\leq i\leq k} and \w[.]{\epsilon\in\{0,1\}}
\end{defn}

\begin{remark}\label{projection}
Using Theorem \ref{thm2.6}, we can easily see that given an inclusion of posets
  \w{i : P \subseteq P'} satisfying the hypotheses of Lemma \ref{lem6.4},
  the induced map on cohomology can be identified with the projection
  \w[.]{\fD\sb{\phi} \to \fD\sb{i\sp{\ast}\phi}}
\end{remark}

\begin{lemma}\label{finitecubes}
Let $K$ be a finite cubical complex. Let $\mX$ be fibrant simplicial category, and
\w{\phi:\mathfrak{C}(K)\to\Po{n-1}\mX} a map. The Andr\'{e}-Quillen cohomology
\w{H\sp{\ast}\sb{Q}(\mathfrak{C}(K), \pi\sb{n}(\mX))} is computed by the cochain
complex \w[.]{C\sp{\ast}(\mathfrak{C}(K))}
\end{lemma}

\begin{proof}
Embed $K$ as a subcomplex in \w{\square\sp{N}} for some large enough $N$,
so we have an inclusion \w[.]{\fC(K)\subseteq\fC(\square\sp{n})}
Without loss of generality, we can replace \w{\fC(K)} by an iterated (homotopy)
pushout of face maps \w{[1]\sp{n} \to [1]\sp{n+1}} (in the Bergner model category),
which we denote by \w[.]{K' \subseteq \fC(K)}

We can think of \w{K'} as a sub-simplicial category of \w[,]{\fC(\square\sp{N})}
generated by a collection of subcategory inclusions
\w{Q\sb{k} : [1]\sp{n} \to [1]\sp{N}} \wb[.]{k = 1, 2, \cdots M} Given
\w[,]{f \in\Mor(K')} write \w{\bSS\sb{f}} for the set of decompositions of $f$, that is,
sequences of morphisms \w{(f\sb{1}, \cdots, f\sb{n})} with
\w[,]{f\sb{i} = Q\sb{k}(g\sb{i}) : a\sb{i} \to a\sb{i+1}}
such that \w[.]{f = f\sb{n} \circ \cdots \circ f\sb{1}}

Let \w{A\sb{f}} denote the image of the composition map
\begin{equation*}
\begin{split}
\bigoplus\sb{(f\sb{1}, \cdots, f\sb{n}) \in \mathbb{S}\sb{f}}&
\left( \pi\sb{n}(\Map\sb{\mX}(\phi(a\sb{1}), \phi(a\sb{2})), f\sb{1} ) \times \cdots
\times \pi\sb{n}(\Map\sb{\mX}(\phi(a\sb{n}), \phi(a\sb{n+1})),f\sb{n} ) \right)  \\
&\to~\pi\sb{n}(\Map\sb{\mX}(\phi(a\sb{1}), \phi(a\sb{n+1})), f)~.
\end{split}
\end{equation*}
Using Lemma \ref{lem6.4}, we can identify \w{\pi\sb{n}(\mX)} with a functor
\w{\Phi:\Tw(K') \to \Sp} given by  \w[.]{f \mapsto K(A\sb{f}, n)}

Consider the full subcategory \w{\Tw\sb{[-, -]}(K')} of \w{\Tw(K')} on morphisms
of the form \w[.]{Q\sb{i}(f)} We claim that the poset inclusion
\w{i : \Tw\sb{[-, -]}(K') \subseteq \Tw(K')} is final.
It suffices to note that for each \w[,]{f\in\Obj(\Tw(K'))} we have
  \w[,]{f = gQ\sb{i}(h)l} so that \w[.]{\Map(Q\sb{i}(h), f) \neq \emptyset}

Thus, we can compute \w{H\sp{\ast}_Q(K', \pi\sb{n}(\mX))} by taking the homotopy
limit of \w[.]{\Phi \circ i} Arguing as in final paragraph of the proof of
Proposition \ref{thm6.5}, we can thus identify \w{H\sp{-m}_Q(K', \pi\sb{n}(\mX))}
with the cohomology of the complex:
\begin{equation*}
\bigoplus\sb
{\stackrel{\ell[a\sb{0},b\sb{0}] = 0}{Q\sb{i}(f\sb{0}): a_0\to b_0}} A\sb{Q\sb{i}(f_0)}
  \xra{\partial\sp{0}} \bigoplus\sb{\stackrel{\ell[a\sb{1}, b\sb{1}] = 0}{
    Q\sb{i}(f\sb{1}) : a\sb{1} \to b\sb{1}}}
 A\sb{Q\sb{i}(f\sb{1})} \to
  \cdots
\bigoplus\sb{\stackrel{\ell[a\sb{n}, b\sb{n}] = 0}{Q\sb{i}(f\sb{n}):a\sb{n} \to b\sb{n}}}
  A\sb{Q\sb{i}(f\sb{n})}~,
\end{equation*}
where \w{\ell[a, b]} denotes the length of an interval in \w[.]{[1]\sp{n}}
But this complex is precisely \w[.]{C\sp{\ast}(K')}
\end{proof}

\begin{thm}\label{thm7.3}
Let $\mY$ and $\mX$ be fibrant simplicial categories, with each composition in $\mY$
monic, and \w{\phi:\mY\to\Po{n-1}\mX} a map. The Andr\'{e}-Quillen cohomology
\w{H\sp{\ast}\sb{Q}(\mY, \pi\sb{n}(\mX))} is computed by the cochain
complex \w[.]{C\sp{\ast}(\mY)}
\end{thm}

\begin{proof}
Let \w{\w{\FinCub\sb{/\fN(\mX)}}} denote the category of finite cubical complexes
over \w[,]{\fN(\mX)} and $J$ that of maps \w{\fC(K) \to \mX} which are
adjoint to those of \w[,]{\FinCub\sb{/\fN\mX}} where \w{K \subseteq \fN(\mX)}
is a finite cubical subcomplex.

It follows from Corollary \ref{cor5.9} and an argument similar to that of
Lemma \ref{lem4.9} that we have an equivalence
$$
\mX \simeq \underset{\phi : (\DK([1]\sp{n}) \to \mX)
  \in(\DK\bone{\ast})\sb{/\mY}\sp{\ND})}{\holim} [1]\sp{n} \simeq
\underset{(\widetilde{\phi}:\fC(K) \to \mX)\in J}{\holim}\fC(K)~.
$$
\noindent Thus, we have
\begin{myeq}[\label{eqholim}]
\mK~ \simeq
\underset{(\widetilde{\phi}:\fC(K) \to \mX)\in J}{\holim}
C\sp{\ast}(\fC(K)) \simeq
\underset{\phi \in(\DK\bone{\ast})\sb{/\mY}\sp{\ND})}{\holim} C\sb{\phi}~,
\end{myeq}
where $\mK$ is a complex whose cohomology computes the values of
\w[.]{H\sp{\ast}\sb{Q}(\mY, \pi\sb{n}(\mX))}
The first  equivalence follows from the preceding paragraph, \wref[,]{eqlimit}
and Lemma \ref{lem6.4}.
The second of these equivalences follows from an argument similar to the proof
of Lemma \ref{lem4.9}.

To identify the homotopy limit on the left hand side of \wref[,]{eqholim} we note that
the category $\square$ of cubes can be given the structure of a Reedy category
by \cite[Corollary 1.17]{DKLSattC}, where the degree decreasing morphisms are
those generated by face maps and degree function
\w{\deg: \Obj(\square) \to \mathbb{N}} is given by
\w[.]{\deg([1]\sp{n}) = n} This induces a Reedy structure on
\w{(\DK\bone{\ast})\sb{/\mY}\sp{\ND}} with
\w[,]{\deg([1]\sp{n} \to \mY) = n}
\w[,]{((\DK\bone{\ast})\sb{/\mY}\sp{\ND})_+ = (\DK\bone{\ast})\sb{/\mY}\sp{\ND}}
and
\w[.]{((\DK\bone{\ast})\sb{/\mY}\sp{\ND})_- =
\{\id_\phi : \phi \in \Obj(( \DK\bone{\ast})\sb{/\mY}\sp{\ND}) \}}

The matching object of \w{\phi: \square\sp{n} \to \mY} is the equalizer
$$
\xymatrix
{
  M\sb{n}(\phi) \ar[r] & \bigoplus\sb{0 \le i \le n, \epsilon = 0, 1}
  C\sb{\phi \circ d\sp{i, \epsilon}} \ar@<1ex>[r]\sp{(d^i)\sp{\ast}}
  \ar@<-1ex>[r]\sb{(d^j)\sp{\ast}} &
  \bigoplus\sb{0 \le i \le n, \epsilon = 0, 1}
  C\sb{\phi \circ d\sp{i, \epsilon} d\sp{j, \epsilon}}~.
}
$$
By Remark \ref{projection}, the natural map \w{\phi \to M\sb{n}(\phi)}
is given by projection onto each factor, and is thus surjective. Therefore, $\phi$ is
a fibration in the model category \w{\Ch\sp{+}(\AbGp)} of non-negatively graded cochain
complexes of abelian groups (see \cite[Definition 4.7]{CCortC}).
We conclude that the diagram
$$
\Psi:
(\DK\bone{\ast})\sb{/\mY}\sp{\ND} \to \Ch^+(\AbGp)\hs\text{given by}\hsm
\phi \mapsto C_\phi
$$
is Reedy fibrant, so the homotopy limit of $\Psi$ can be computed as an ordinary limit.
\end{proof}

\begin{corollary}\label{cor7.1}
Let $\mX$ and $\mY$ be fibrant simplicial categories such that composition in $\mY$
is monic, and consider the lifting problem
\mydiagram[\label{liftingdiagram}]{
& \Po{n}\mX \ar[d] \\
  \mY \ar[r]\sb{\phi} \ar@{.>}[ur] & \Po{n-1}\mX~.
  }
\noindent The obstruction to having a lift in \ref{liftingdiagram} is
given by a class \w[.]{\beta \in H\sp{1}(C\sp{\ast}(\mY))}
\end{corollary}

\begin{proof}
Since \w{\Po{n}\mX} is an \wwb{n+1}category, this follows from Theorem \ref{thm2.7}.
\end{proof}

\begin{remark}\label{rcubcx}
By Lemma \ref{lem5.5}, the hypothesis $\mY$ is monic will hold if $\mX$ is fibrant
and cofibrant. Moreover, if $\widehat{\mY}$ is a cubical \emph{complex}, it is provided
with a colimit decomposition which is actually a homotopy colimit, so if
\w{\widehat{\mY}\to\mY} is a fibrant replacement (in \w[),]{\sCat} the same
homotopy colimit description holds for $\mY$, and therefore provides a more
explicit presentation of the cochain complex \w[.]{C\sp{\ast}(\mY)}
\end{remark}

%
%
\sect{The lifting obstruction for the boundary of a cube}
\label{ccccc}

Proposition \ref{thm2.4} and Theorems \ref{thm5.8} and \ref{thm7.3}
suggest that the key to understanding the $k$-invariants of a
simplicially enriched category $\mX$ is to analyze them for the
cubes and cubical complexes in $\mX$. Here we do so for the
simplest non-trivial case: the boundary of a cube. In the next two
sections we explain how this example can be used for various
applications.

\begin{mysubsection}{\w{\bone{n}} as a simplicial category}
\label{sosc}
Note that for each \w[,]{n\geq 0} the category \w{\bone{n}} may be identified with
the lattice of subsets of \w[,]{\bbn=\{1,2,\dotsc,n\}} which we shall think of
as an $n$-dimensional cube \w{I\sp{n}} with vertices labelled by the characteristic
function \w{J=(\var{1},\var{2},\dotsc,\var{n})} of a subset of  $\bbn$ (with
each \w[).]{\var{i}\in\{0,1\}} Thus \w{(0,0,\dotsc,0)} is \w{J\sb{\emptyset}} and
\w[.]{(1,1,\dotsc,1)} is \w{J\sb{\bbn}}
Write \w{|J|:=n-\var{1}+\var{2}+\dotsc+\var{n}} for the number of zeros in $J$,
so \w{|J\sb{\emptyset}|=n} and \w[.]{|J\sb{\bbn}|=0} The morphisms in
\w{\bone{n}} are given by \w{J\preceq J'} whenever
\w{\var{i}\leq\varp{i}} for all \w[.]{1\leq i\leq n}

If $X$ is an \wwb{\infty,1}category, an \emph{$n$-cube in} $X$ is indexed by
\w[.]{B(\bone{n})} In particular, we can use the Dwyer-Kan resolution
\w{\DK(\bone{n})} as a cofibrant model for \w{B(\bone{n})} in \w[.]{\sCat}

Recall that the $n$-\emph{permutohedron} \w{\cP\sb{n}} is the convex hull of the
\w{(n+1)!} points in \w{\RR\sp{n+1}} obtained by permuting the coordinates of
\w{(x\sb{0},\cdots,x\sb{n})} for \w{n+1} distinct real numbers
\w[.]{\{x\sb{0},\cdots,x\sb{n}\}}
Thus the $0$-permutohedron is a point, the $1$-permutohedron is an interval,
and the $2$-permutohedron is a hexagon.

Denoting the standard triangulation of \w{\cP\sb{n}} by \w{P\sb{n}}
(cf.\ \cite[\S 5.3]{BMeadS}), we have:
\end{mysubsection}

\begin{prop}\label{pperm}
For every \w{J,J'\in \bone{n}} with \w[:]{J\prec J'}
\begin{myeq}[\label{eqpermut}]
\Map\sb{\DK(\bone{n})}(J, J')~\cong~ P\sb{|J|-|J'|-1}~.
\end{myeq}
\end{prop}

\begin{proof}
There is a functor \w{U:\bone{n}\to\Drn{n-1,-1}\op} (see \S \ref{snac}),
bijective on the set of arrows and unique up to automorphism of \w[,]{\bone{n}}
with \w{U(J)=\bmm} if \w[.]{|J|=m+1} This implies that the mapping spaces in
\w{\DK(\bone{n})} are the components of the corresponding mapping spaces in
\w[,]{\DK(\Drn{n-1}\op)} and the result follows from \cite[Proposition 5.6]{BMeadS}.
\end{proof}

\begin{corollary}\label{cperm}
In any simplicial category $\mX$ and \w[,]{n\geq 0} the $k$-invariants for any
$n$-cube \w{f:\DK(\bone{n})\to\mX} vanish.
\end{corollary}

\begin{proof}
All the mapping spaces in \w{\DK(\bone{n})} are weakly contractible, so the same is
true of their images under $f$ in $\mX$.
\end{proof}

\begin{mysubsection}{The $k$-invariant of a cube}
\label{skinvcube}
For all pairs of the form \w{J\prec J'} in \w{\Obj{\bone{n}}=\Obj{\DK(\bone{n})}}
except \w[,]{J\sb{\emptyset}\prec J\sb{\bbn}} the simplicial set
\w{\Map\sb{\DK(\bone{n})}(J, J')} is of dimension \w[.]{m\leq n-1}
We may assume $\mX$ is fibrant, and factor \w{f:\DK(\bone{n})\to\mX} as
\w{\DK(\bone{n})\xra{j}\mX\sb{f}\xra{p}\mX} \wwh a trivial cofibration $j$
followed by a fibration in \w[.]{\sCat} Since \w{\DK(\bone{n})} is cofibrant,
\w{\mX\sb{f}} is the (weakly contractible) fibrant and cofibrant sub-simplicial
category of $\mX$ generated by the image of $f$.

Even though the $k$-invariants of the $n$-cube $f$ are trivial, we still need to
represent them at the space level, as in Theorem \ref{thm2.7}.
In particular, for \w[,]{J\sb{\emptyset}\prec J\sb{\bbn}} we need a representative
of a class in \w[,]{\pi\sb{n-1}(M,\phi')} where
\w{M:=\Map\sb{\mX}(f(J\sb{\emptyset}),\,f(J\sb{\bbn}))} and \w{\phi'=f(\phi)} for
  \w{\phi:J\sb{\emptyset}\to J\sb{\bbn}} the unique map in \w{\bone{n}}
(thought of as a vertex in \w[).]{\DK(\bone{n})}

Let \w{\alpha\in M\sb{n-1}} be a simplex representing the \wwb{n-1}permutohedron
$P$ \wh that is, the top cell of
\w[:]{\Map\sb{\DK(\bone{n})}(J\sb{\emptyset},J\sb{\bbn})}
this exists because \w{\mX\sb{f}} is fibrant, so $M$ is a Kan complex, but it
has non-trivial boundary $\beta$ (the restriction of $f$ to \w[).]{\partial P}
To remedy this, we follow \cite[\S 4]{BDreC} or \cite[\S 8]{BPaolT}:

For every \w{0\leq k\leq n-1} and $k$-simplex $\sigma$ in the nerve
of \w[,]{f\sb{\#}(\bone{n})\subseteq\ho\mX} choose a fixed representative \w{s(\sigma)}
in \w[.]{\mX\sb{f}} Note that once we have made our initial choice \w{s(g)} for each
edge $g$ of the cube $f$, we have a contractible set of choices for the remaining
vertices and higher simplices, corresponding to the various composites of the edges
$g$ and the witnesses for these composites.
The non-degenerate \wwb{n-1}simplices \w{(\tau\sb{i})\sb{i=1}\sp{N}} form a
triangulation of \w[.]{f[P]} For each such \w{\tau\sb{i}} indexed by the composable
sequence \w{(g\sp{i}\sb{1},\dotsc, g\sp{i}\sb{n})} in \w[,]{\ho\mX} let $v$
be the vertex of \w{\tau\sb{i}} indexed by the composite
\w{s(g\sp{i}\sb{1}\circ\dotsc\circ g\sp{i}\sb{k})} (which, as the barycenter of
$P$, is common to all the simplices \w[).]{\tau\sb{i}}
See Figure \ref{eqfatsquare} below.

Let \w{\widehat{\tau\sb{i}}} denote the prism on the \wwb{n-1}face
of \w{\tau\sb{i}} opposite $v$, with
\w{q:\widehat{\tau\sb{i}}\to\tau\sb{i}} the quotient map
collapsing the top facet of \w{\widehat{\tau\sb{i}}} to the point
\w[.]{v\sp{i}} The union
\w{T:=\bigcup\sb{i=1}\sp{N}\,\widehat{\tau\sb{i}}} of these prisms
(with identifications along the sides corresponding to those of
the triangulation of $P$) forms a collaring of $P$, in the sense that $P$ is a
deformation retract of \w[,]{\widehat{P}:=P\cup T} with the retraction collapsing each
prism to its base. The extension $\widehat{\alpha}$ of $\alpha$ to $\widehat{P}$ is
defined by collapsing the top \w{A'} of each prism \w{\widehat{\tau\sb{i}}} in $T$
to a point, yielding a cone on the bottom $A$ of the prism, which is mapped to the
image of its reflection inside $P$
(the cone on $A$ with cone point the barycenter of $P$). In particular, \w{A'} is sent
to the basepoint \w{s(\phi)\in\mX\sb{f}} corresponding to the chosen representative
of the composite of the edges of the cube $f$. Thus the map $\widehat{\alpha}$, the
image of $\widehat{P}$ under \w[,]{j:\DK(\bone{n})\to\mX\sb{f}} is a \emph{topological
  representative} for a (necessarily trivial) element in \w[,]{\pi\sb{n-1}M} and the
latter is the value of the \wwb{n-2}th $k$ invariant for the $n$-cube $f$ in $\mX$
(or equivalently, for the fibrant simplicial category \w[).]{\mX\sb{f}}
\end{mysubsection}

\begin{example}\label{egtwocube}
When \w[,]{n=2} the cohomology groups we obtain are generally not
abelian (unless $\mX$ is \emph{linear} in the sense of \cite[\S
2-4]{BDreC}). However, the above description still makes sense. In
fact, if we think of the quasi-category description of a commuting
square in the form of the central square in Figure
\wref[,]{eqfatsquare} we depict the two $2$-simplices
\w{\tau\sb{1}=(h\sb{1}g\sb{1})} and \w{\tau\sb{2}=(h\sb{2}g\sb{2})}
which constitute the triangulation of the square on the outside,
as a collaring. For simplicity we assume that the given edges of
the square are precisely the chosen representatives
\w[,]{s[g\sb{1}]} and so on \wh otherwise we need to interpolate
an appropriate homotopy between the global choice \w{s[g\sb{1}]}
and the value \w{g\sb{1}} assigned by $f$.
\mytdiag[\label{eqfatsquare}]{
&&&\\
&  f(J\sb{00}) \ar `l[ld] `d[dd]\sb{\phi=s([h\sb{1}]\circ[g\sb{1}])} `r[r] [rd]
  \ar `u[ru] `r[rr]\sp(0.5){\phi=s([h\sb{2}]\circ[g\sb{2}])} `d[d] [rd]
  \ar[d]\sb{s[g\sb{1}]} \ar[r]\sp{s[g\sb{2}]} &
f(J\sb{01})\ar@{=>}[ru]\sb{s(\tau\sb{2})}\ar[d]\sp{s[h\sb{2}]} & \\
& f(J\sb{10}) \ar[r]\sb{s[h\sb{1}]}
\ar@{=>}[ld]\sb{s(\tau\sb{1})} \ar@{=>}[ru]\sb{\alpha} & f(J\sb{11})  &  \\
  &&&
}
\noindent In the corresponding simplicial category $\mX$, this appears as the following
concatenation of $1$-simplices in \w[:]{M:=\Map\sb{\mX}(f(J\sb{00}),\,f(J\sb{11}))}
\mytdiag[\label{eqflatsquare}]{
  \phi & s[h\sb{1}]\circ s[g\sb{1}] \ar[l]\sb{s[\tau\sb{1}]} \ar[r]\sp{\alpha} &
s[h\sb{2}]\circ s[g\sb{2}] \ar[r]\sp{s[\tau\sb{2}]}& \phi
  }
\noindent which yields a loop in $M$ based at $\phi$, as required.  To see that
it is null-homotopic, note that the concatenation of \w{s[\tau\sb{1}]} and
\w{s[\tau\sb{2}]} along $\phi$ is homotopic to $\alpha$, since \w{\mX\sb{f}}
is weakly equivalent to the image of $f$, which has a single such $1$-simplex.
\end{example}

\begin{mysubsection}{The boundary of a cube}
\label{sbcube}
Even though each simplicial category $\mX$ is the homotopy colimit of its cubes,
this does not imply, of course, that the $k$-invariants for $\mX$ are trivial.
The most important example is the boundary of an $n$-cube.
\end{mysubsection}

\begin{defn}\label{dbcube}
Given a fibrant simplicial category $\mX$ and a map
$$
f:\DK(\sk{n-1}\bone{n})\to\mX~,
$$
\noindent factor $f$ as a trivial cofibration
\w{j:\DK(\sk{n-1}\bone{n})\to\mX\sb{f}} followed by a fibration
\w[.]{\phi:\mX\sb{f}\to\mX}
Let
$$
P:=\Map\sb{\DK(\bone{n})}(J\sb{\emptyset},\,J\sb{\bbn})\hs\text{and}\hs
M:=j(P)=\Map\sb{\mX}(f(J\sb{\emptyset}),f(J\sb{\bbn}))~,
$$
\noindent with base-point \w{\psi:f(J\sb{\emptyset})\to f(J\sb{\bbn})}
a fixed representative of the composite of the edges of the cube.

The corresponding \emph{facet class} \w{\beta\sb{f}\in\pi\sb{n-2}(M,\psi)}
is obtained by applying the universal property of the homotopy colimit to the
representatives \w{\alpha\sb{\rho}} constructed in \S \ref{skinvcube} above,
for each of the facets $\rho$ of \w{\DK(\sk{n-1}\bone{n})} (or equivalently, of
\w[),]{\mX\sb{f}} which are themselves \wwb{n-1}cubes in \w[.]{\mX\sb{f}}

Note that the collarings \w{T\sb{\rho}} of each permutohedron
\w{T\sb{\rho}} map into $\mX$ by the fixed representatives \w[,]{s(\sigma)}
which do not depend on $\rho$.
\end{defn}

\begin{thm}\label{tpbclass}
Given a map \w{f:\DK(\sk{n-1}\bone{n})\to\mX} into a fibrant simplicial category,
with \w{\phi:\mY:=\mX\sb{f}\to\mX} a fibrant replacement as above,
the pullback $\beta$ of the \wwb{n-2}nd $k$-invariant for $\mX$ along $\phi$
is represented by the facet class \w[.]{\beta\sb{f}}
\end{thm}

\begin{proof}
  This is the class $\beta$ of Corollary \ref{cor7.1}, and since $\mY$ is the
  homotopy colimit of the images of the facets of \w[,]{\bone{n}}
  by Theorem \ref{thm4.1},
$\beta$ is determined by the universal property of the colimit form the values
on the individual facets, which are given by Proposition \ref{thm6.5} and
the discussion in \S \ref{skinvcube}. Note that different choices of
\w{s(\sigma)} there may yield different values in \w[,]{\pi\sb{n-2}(M,\psi)}
but the resulting cocycles will be cohomologous (see \cite[Lemma A.1]{BDreC}).
\end{proof}

\begin{example}\label{egboundcube}
For \w[,]{n=3} we can extend the description of Example \ref{egtwocube} as
follows: in the quasi-category description of the boundary of a $3$-cube we have
Figure \wref[,]{eqboundcube} where the homotopies on the three hidden facets
are dashed.
\myudiag[\label{eqboundcube}]{
%
%
f(J\sb{000}) \ar@<1ex>[rr]\sp{s[g\sb{2}]} \ar|(.7){\hole}[rd]\sp{s[g\sb{1}]}
\ar[dd]\sb{s[g\sb{3}]} &&
f(J\sb{010}) \ar[rd]\sp{s[h\sb{21}]} \ar|(.5){\hole}[dd]\sb(0.7){s[h\sb{23}]} &&  \\
& f(J\sb{100}) \ar[rr]\sp(0.65){s[h\sb{12}]} \ar[dd]\sp(0.3){s[h\sb{13}]}
\ar@{=>}[ru]\sb{\alpha\sb{0}} &&
f(J\sb{110}) \ar[dd]\sp{s[k\sb{3}]} \\
%
%
f(J\sb{001}) \ar|(.5){\hole}[rr]\sp(0.7){s[h\sb{32}]} \ar[rd]\sb{s[h\sb{31}]}
\ar@{=>}[ru]\sb{\alpha\sb{2}} \ar@/^2.1em/@{==>}[rruu]\sp(0.6){\alpha\sb{1}}
&& f(J\sb{011}) \ar|(.23){\hole}[rd]\sp{s[k\sb{1}]} \ar@{==>}[ru]\sp{\beta\sb{2}} & \\
& f(J\sb{101}) \ar[rr]\sb{s[k\sb{2}]} \ar@{==>}[ru]\sp{\beta\sb{3}}
\ar@<1ex>@/_2.1em/@{=>}[rruu]\sb(0.4){\beta\sb{1}} &&
f(J\sb{111})
}
\noindent Each of the facets should then be collared as in \wref[,]{eqfatsquare}
yielding six separate loops in the appropriate mapping spaces in $\mX$ as
in \wref[.]{eqflatsquare}

However, in order to obtain the value of the lifting obstruction
\w{\beta\in\pi\sb{1}(M,\phi)} for \w{M:=\Map\sb{\mX}(f(J\sb{000}),\,f(J\sb{111}))}
and
$$
\phi=s[k\sb{3}\circ h\sb{12}\circ g\sb{1}]=
s[k\sb{3}\circ h\sb{21}\circ g\sb{2}]=\dotsc
$$
(the unique homotopy class in \w[),]{[f(J\sb{000}),f(J\sb{111}]} we have to pull back
or push forward these loops so they all lie in \w[,]{M\sb{1}} and then concatenate
the result, using  the Homotopy Addition Theorem
(cf.\ \cite[Theorem 2.4]{CurtS}), as in Figure \wref[.]{eqpermutahedron}
\myvdiag[\label{eqpermutahedron}]{
&&
s[k\sb{3}]\circ s[h\sb{12}]\circ s[g\sb{1}]\ar[dll]\sb{s[k\sb{3}]\sb{\#}\alpha\sb{0}}
&& \\
  s[k\sb{3}]\circ s[h\sb{21}]\circ s[g\sb{2}] &&&&
  s[k\sb{2}]\circ s[h\sb{13}]\circ s[g\sb{1}]
  \ar[ull]\sb{s[g\sb{1}]\sp{\#}\beta\sb{1}} \\
  &&&&\\
  s[k\sb{1}]\circ s[h\sb{23}]\circ s[g\sb{2}] \ar[uu]\sp{s[g\sb{2}]\sp{\#}\beta\sb{2}}
  &&&& s[k\sb{2}]\circ s[h\sb{31}]\circ s[g\sb{3}]
  \ar[uu]\sb{s[k\sb{2}]\sp{\#}\alpha\sb{2}}\\
  && s[k\sb{1}]\circ s[h\sb{32}]\circ s[g\sb{3}]
  \ar[urr]\sp{s[k\sb{1}]\sp{\#}\alpha\sb{1}}
  \ar[ull]\sb{s[g\sb{3}]\sp{\#}\beta\sb{3}} &&
}
\noindent Here we have omitted the left and right collarings in the six versions of
\wref[:]{eqflatsquare} for example, the top left edge
in \wref{eqpermutahedron} really should take the form:
\mysdiag[\label{eqflathexagon}]{
\phi &  s[k\sb{3}]\circ s[h\sb{21}\circ g\sb{2}] \ar[l]\sb(0.65){s[\sigma\sb{1}]} &
s[k\sb{3}]\circ s[h\sb{21}]\circ s[g\sb{2}]
\ar[l]\sb(0.49){s[k\sb{3}]\sb{\#}s[\tau\sb{1}]} &
\ar[l]\sb(0.25){s[k\sb{3}]\sb{\#}\alpha\sb{0}} \\
  & s[k\sb{3}]\circ s[h\sb{12}]\circ s[g\sb{1}]
     \ar[r]\sp(0.5){s[g\sb{1}]\sp{\#}s[\tau\sb{2}]}&
   s[k\sb{3}\circ h\sb{12}]\circ s[g\sb{1}] \ar[r]\sp(0.6){s[\sigma\sb{2}]} & \phi~,
}
\noindent for \w[,]{\tau\sb{1}=(h\sb{21},g\sb{2})}
\w[,]{\tau\sb{2}=(k\sb{3},h\sb{12})}
\w[,]{\sigma\sb{1}=(k\sb{3},h\sb{21}\circ g\sb{2})}
and \w[.]{\sigma\sb{2}=(k\sb{3}\circ h\sb{21},g\sb{1})}

Similarly, the left vertical edge arrow really ends in
\mytdiag[\label{eqrightthexagon}]{
  s[k\sb{3}]\circ s[h\sb{21}]\circ s[g\sb{2}]
  \ar[r]\sp{s[g\sb{2}]\sp{\#}s[\tau\sb{3}]}&
s[k\sb{3}\circ h\sb{21}]\circ s[g\sb{2}]\ar[r]\sp(0.6){s[\sigma\sb{3}]} & \phi~,
}
\noindent for \w{\tau\sb{3}=(k\sb{3},h\sb{21})}
and \w[.]{\sigma\sb{3}=(k\sb{3}\circ h\sb{21},g\sb{2})}

However, we can then use the inner horn condition in $\mX$, or the Kan condition
in $M$, to choose successive fill-ins as follows:
\mywdiag[\label{eqfillins}]{
  &&  s[k\sb{3}]\circ s[h\sb{21}]\circ s[g\sb{2}] \\
  &  s[k\sb{3}]\circ s[h\sb{21}\circ g\sb{2}]
  \ar[ru]\sp{s[k\sb{3}]\sb{\#}s[\tau\sb{1}]\sp{-1}} & \\
  \phi \ar[ru]\sp(0.4){s[\sigma\sb{1}]\sp{-1}} && \\
  &  s[k\sb{3}\circ h\sb{21}]\circ s[g\sb{2}]
  \ar[lu]\sp(0.6){s[\sigma\sb{3}]} \ar@{-->}[uu]\sb{\omega}
  \ar@{-->}[uuur]\sp{\psi}&\\
  && s[k\sb{3}]\circ s[h\sb{21}]\circ s[g\sb{2}]
  \ar[lu]\sp{s[g\sb{2}]\sp{\#}s[\tau\sb{3}]} \ar@{-->}[uuuu]\sp{\theta}~,
}
\noindent where $\theta$ need not be homotopic to the identity, since the
choices in \wref{eqboundcube} are all coherent only when the cube has a fill in.
Thus in the simplified diagram \wref{eqpermutahedron} we must actually add
a self-map of each vertex (corresponding to $\theta$ in \wref[).]{eqboundcube}
Note that we need to leave the full version
at one vertex at least if we want a $k$-invariant value in
\w[;]{\pi\sb{1}(M,\phi)} otherwise we must replace $\phi$ by another basepoint
in the same component of $M$.
\end{example}

\begin{mysubsection}{Cubical complexes}
\label{scubcx}
Given a finite cubical complex $K$ in a simplicial category $\mX$, the discussion
in \S \ref{sbcube} yields a full description of the pullback of the
$k$-invariants of $\mX$ to $K$, by induction on the cubical skeleta \w[:]{\sk{n}K}
the passage from \w{\sk{n-1}K} to \w{\sk{n}K} involves filling in certain boundaries
of $n$-cubes in \w[,]{\sk{n-1}K} while forming new boundaries of \wwb{n+1}cubes.
This results in killing the corresponding values in \w{\pi\sb{n-2}\mX}
for \w{k\sb{n-3}} of $K$, while generating new values in
\w{\pi\sb{n-1}\mX} for \w[.]{k\sb{n-2}}

The restriction to (finite) cube complexes, rather than arbitrary cubical sets,
is important, because in this case their structure provides a homotopy colimit
decomposition in terms of ordinary colimits (see Remark \ref{rcubcx}).
\end{mysubsection}

%
%
\sect{$k$-invariants and higher structure}
\label{ckihs}

The local system \w{\pi\sb{n}\mX} of a simplicial category $\mX$
is determined by its homotopy category, so in the principle
the $n$-th $k$-invariant for $\mX$ encodes all the strictly \wwb{n+2}nd order
information in $\mX$ (beyond the information of order \w{\leq n+1} encoded in
the \wwb{n-1}st Postnikov section \w[).]{\Po{n-1}\mX} In order to show that this
statement is not completely vacuous, we briefly describe three simple examples:

\begin{mysubsection}{Differentials in spectral sequences}
\label{sdss}
Spectral sequences occur in many versions, but the most general, for our purposes,
appears to be that of a simplicial space: if \w{\Wd} is a simplicial object and $\bY$
is a homotopy cogroup object in a simplicial category $\mX$, the corresponding
\emph{homotopy spectral sequence} (see \cite[Theorem B.5]{BFrieH} and \cite[\S 8]{DKStB})
has
\begin{myeq}[\label{eqssss}]
E\sp{2}\sb{n,k}~\cong~\pi\sb{n}\sp{h}\pi\sb{k}\sp{v}\map\sb{\mX}(\bY,\Wd)~\Longrightarrow~
\pi\sb{n+k}\map\sb{\mX}(\bY,\hocolim\Wd)
\end{myeq}
\noindent The construction can be extended to any reasonable model of
\wwb{\infty,1}categories (see \cite[\S 1]{BMeadS}).

In \cite{BMeadS}, we provide the following description of the differentials:
if we represent an element in \w{E\sp{1}\sb{n,k}} by a map
\w[,]{f:\Sigma\sp{k}\bY\to\bW\sb{n}} then it survives to the \ww{E\sp{r}}-term if and
only if it can be extended to a homotopy coherent diagram
\myrdiag[\label{eqerrep}]{
\Sigma\sp{k}\bY  \ar@/^1pc/[rr]\sp{0}\sb{\vdots}
  \ar@/_1pc/[rr]\sb{0}\ar[dd]\sb{f}&&
0  \ar@/^1pc/[rr]\sb{\vdots}  \ar@/_1pc/[rr] \ar[dd] && 0 \ar[dd] & \cdots \cdots &
0 \ar[dd]\\
&& && & & \\
\bW\sb{n} \ar@/^1pc/[rr]\sp{d\sb{0}}\sb{\vdots}\ar@/_1pc/[rr]\sb{d\sb{n}}&&\bW\sb{n-1}
\ar@/^1pc/[rr]\sp{d\sb{0}}\sb{\vdots}  \ar@/_1pc/[rr]\sb{d\sb{n-1}} && \bW\sb{n-2} &
\cdots \cdots & x\sb{n-r+1}
}
\noindent in $\mX$. Assuming \w{\Wd} is Reedy fibrant, by adjunction this yields a map
from \w{\DK(\Drn{n, n-r+1}\op)} (see \S \ref{snac}) to the truncation of
\w[.]{\map\sb{\mX}(\bY,\Wd)}
Here we use the fact that the mapping space in the former from \w{\bn} to \w{\bnrp} 
is a disjoint union of copies of \w{\map\sb{\DK(\Drn{r,-1}}(\br,\bmo)} (see
\cite[Proposition 5.6]{BMeadS}), but only one of the copies is
relevant here, by \cite[Corollary 6.10]{BMeadS}.

In particular, we obtain from the diagram \wref{eqerrep}
a map $\varphi$ from the boundary \w{\partial P\sp{r}} of the appropriate permutohedron
to \w{\map\sb{\mX}(\Sigma\sp{k}\bY,\bW\sb{n-r})} (see \S \ref{sosc}) \wh or,
by another adjunction, a map \w[.]{\phi:\Sigma\sp{k+r-1}\bY\to\bW\sb{n-r}}

We show in \cite[\S 6]{BMeadS}) that the representatives in \w{E\sp{1}\sb{n-r,k+r-1}} of
the value of the differential \w{d\sp{r}[f]\in E\sp{r}\sb{n-r,k+r-1}}
are precisely the homotopy classes of the maps $\phi$ associated to various
choices of the extension \wref[.]{eqerrep} Moreover, the class of $\phi$ is precisely
the value of the $k$-invariant for the map \w{F:\partial \bone{r+1}\to\mX}
obtained by precomposing $\varphi$ with the map
\w{U\sb{\ast}:\DK(\bone{r+1})\to\DK(\Drn{r,-1}\op)} from the proof of
Proposition \ref{pperm}.
Thus in this case extracting the \wwb{r+1}st order information contained in \w{d\sp{r}}
from the \wwb{r-1}st $k$-invariant for $\mX$ is completely straightforward.

The differentials in the dual homotopy spectral sequence of a cosimplicial
object are treated in \cite[\S 7]{BMeadL}; those in the more common stable
spectral sequences of a tower (or sequence of objects) are special cases of
the two unstable versions, in the sense that any differential in the stable Adams
spectral sequence, say, may be identified with the corresponding differential in the
unstable Adams spectral sequence in the stable range.
Of course, the spectral sequences as a whole are different.
\end{mysubsection}

\begin{mysubsection}{Higher order homotopy operations}
\label{shoho}
The first known examples of second order homotopy invariants were Massey products
(in \cite{MassN}) and Toda brackets (in \cite{TodG}). In \cite{BBSenT}, a general
definition of $n$-th order Toda brackets was given (see also \cite{GWalkL}):
these were described as the last obstruction to extending a certain type of
homotopy-coherent diagram in a simplicially enriched category $\mX$, indexed by the
boundary of an \wwb{n+1}cube to the interior, and as such can again be recovered
directly from the corresponding $k$-invariant (see \S \ref{sbcube}).

For example, the original (secondary) Toda bracket corresponds to
\myxdiag[\label{eqtoda}]{
\bX\ar[rr]\ar[rd]\sb{f} \ar[dd] && \ast \ar[rd] \ar|(.5){\hole}[dd]& \\
& \bY\ar[rr] \ar[dd]\sb(0.3){g} && \ast \ar[dd]\\
\ast \ar|(.5){\hole}[rr] \ar[rd] && \ast  \ar[rd] & \\
& \bZ \ar[rr]\sb{h} && \bW
}
\noindent (with the homotopies along each facet \wh actually, nullhomotopies \wh
understood).

A more general approach uses the differentials in the homotopy spectral sequence of
certain types of simplicial space \w[,]{\Wd} as described in \S \ref{sdss},
to \emph{define} higher order operations in terms of maps from
\w{\DK(\Drn{n+1,m}\op)} to \w{M:=\map\sb{\mX}(\bS{k},\bW\sb{m})}
(see \cite{BBSenHS}).

For example, if we think of \wref{eqerrep} as the bottom truncated simplicial space
extended one dimension to the left by setting \w[,]{\bW\sb{n+1}=\Sigma\sp{k}\bY=\bS{k}}
with \w{d\sb{i}:\bW\sb{n+1}\to\bW\sb{n}} given by $f$ for \w{i=0} and by $\ast$
for \w[,]{i>0} then for \w{r=2} the permutohedron \w{P\sp{2}} is a hexagon, and the map
\w{\partial P\sp{2}\to M} is given by

\vspace{1.0cm}\quad

\myfigure[\label{eqhexagon}]{
\begin{picture}(100,180)(-100,-140)
%
%
\put(95,85){\line(-3,-2){30}}
\put(49,79){\scriptsize $d\sb{0}H\sb{01}$}
\put(97,87){\circle*{3}}
\put(83,92){\small $d\sb{0}d\sb{0}d\sb{0}$}
%
%
\put(100,85){\line(3,-2){30}}
\put(127,79){\scriptsize $H\sb{01}d\sb{0}$}
\put(132,63){\circle*{3}}
\put(137,60){\small $d\sb{0}d\sb{1}d\sb{0}$}
%
%
\put(132,60){\line(0,-1){30}}
\put(137,42){\scriptsize $d\sb{0}H\sb{02}$}
\put(132,28){\circle*{3}}
\put(135,20){\small $d\sb{0}d\sb{0}d\sb{2}$}
%
%
\put(130,26){\line(-3,-2){30}}
\put(116,8){\scriptsize $H\sb{01}d\sb{2}$}
\put(97,5){\circle*{3}}
\put(80,-6){\small $d\sb{0}d\sb{1}d\sb{2}$}
%
%
\put(64,28){\circle*{3}}
\put(26,25){\small $d\sb{0}d\sb{1}d\sb{1}$}
\put(65,26){\line(3,-2){30}}
\put(48,7){\scriptsize $d\sb{0}H\sb{12}$}
%
%
\put(64,63){\circle*{3}}
\put(28,60){\small $d\sb{0}d\sb{0}d\sb{1}$}
\put(64,60){\line(0,-1){30}}
\put(35,42){\scriptsize $H\sb{01}d\sb{1}$}
\end{picture}
}

\vspace{-4.7cm}\quad

\noindent Here \w{H\sb{ij}:d\sb{i}d\sb{j}\sim d\sb{j+1}d\sb{i}}
\wb{i<j} is a homotopy for the simplicial identity (these are
actually only needed on the right in \wref[,]{eqhexagon} since we
may assume \w{\Wd} is a strict simplicial space). In particular,
if we require each \w{\bW\sb{n}} to be a wedge of spheres, these
are higher order \emph{homotopy} operations in the strict sense,
in that they are modelled on ordinary homotopy groups in the
homotopy category of spaces, and the primary operations between
them.
\end{mysubsection}

\begin{example}\label{egswp}
If \w[,]{\bS{p}} \w[,]{\bS{q}} and \w{\bS{r}} are simply-connected
pointed spheres and
\begin{myeq}[\label{eqfatwedge}]
\bT{1}=\bT{1}(\bS{p},\bS{q},\bS{r})~:=~\{(x,y,z)\in\bS{p}\times\bS{q}\times\bS{r}~|\
  x=\ast~\text{or}\ y=\ast~\text{or}\ z=\ast\}
\end{myeq}
\noindent is the corresponding fat wedge, we have a homotopy-commuting cube
\myydiag[\label{eqhwp}]{
  \bS{p+q+r-2}\ar[rr]\sp{(-1)\sp{p+q+1}[\iot{p+q-1},\iot{r}]}
  \ar[rd]\sp(0.6){(-1)\sp{p+r+1}[\iot{p+r-1},\iot{q}]}
  \ar[dd]\sb{(-1)\sp{q+r+1}[\iot{q+r-1},\iot{p}]} &&
  \bS{p+q-1}\vee\bS{r} \ar[rd]\sp{[\iot{p},\iot{q}]\bot\iot{r}}
  \ar|(.5){\hole}[dd]\sb(0.3){[\iot{p},\iot{q}]\bot\iot{r}} & \\
  & \bS{p+r-1}\vee\bS{q}\ar[rr]\sb(0.75){[\iot{p},\iot{r}]\bot\iot{q}}
  \ar[dd]\sb(0.25){[\iot{p},\iot{r}]\bot\iot{q}} &&
  \bT{1}''' \ar[dd]\sp{\iot{}'''}\\
  \bS{q+r-1}\vee\bS{p} \ar|(.5){\hole}[rr]\sp(0.7){[\iot{q},\iot{r}]\bot\iot{p}}
  \ar[rd]\sb{[\iot{q},\iot{r}]\bot\iot{p}} && \bT{1}'' \ar[rd]\sp{\iot{}''} & \\
& \bT{1}' \ar[rr]\sb{\iot{}'} && \bT{1}~.
}
\noindent Here \w[,]{\bT{1}'} \w[,]{\bT{1}''} and \w{\bT{1}'''} denote
the subcomplexes of \w{\bT{1}} in which we omit the top cell of
\w[,]{\bS{p}\times\bS{q}} \w[,]{\bS{p}\times\bS{r}} and
\w[,]{\bS{q}\times\bS{r}} respectively, and the various maps $\iot{}$ are the obvious
fundamental class or inclusion in each case. The value of the corresponding
$k$-invariant in \w{\pi\sb{1}\map(\bS{p+q+r-2},\,\bT{1})} is the secondary Whitehead
product of \cite{HardHW} (see also \cite{GPorHO}).
\end{example}

\begin{remark}\label{rgenhho}
A yet more general definition of higher operations, based on an arbitrary
(finite, directed) indexing category $\J$, was given in \cite{BMarkH}.
It seems likely that these too can be derived directly from the $k$-invariants
of appropriate cubical complexes, by using a cubical replacement of \w[.]{\DK(\J)}
\end{remark}

\begin{mysubsection}{Classifying diagrams}
\label{scdiag}
Another important question in this context is classifying the realizations
of a given diagram \w{X:\Gamma\to\ho\cS} \wwh that is, the various lifts
\w{\widetilde{X}:\Gamma\to\cS} (if they exist). This was the original motivation
for introducing the cohomology of simplicial categories in \cite{DKanO,DKSmitO}.
In fact, the classification of all such realizations (as $\Gamma$ varies)
provides a full homotopy invariant for \wwb{\infty,1}categories
(see \cite{RenaP,FKRoveM}).

A global approach to this problem is to try to calculate the set of components
of the moduli space \w{\M(\Gamma)} of all such realizations \wh possibly
by using the spectral sequence associated to an appropriate tower of
moduli spaces of ``partial realizations'' (see \cite{DKanC}). Alternatively,
one could try to identify a particular realization, which we think of
as a map of simplicial categories \w[,]{\widehat{X}:\DK(\Gamma)\to\cS}
by induction over the skeleta \w[.]{\sk{n}\DK(\Gamma)} When the indexing
category $\Gamma$ is finite, \w{\DK(\Gamma)} is $n$-dimensional, say, and so there
is a last obstruction for distinguishing between two different lifts
\w{\widehat{X}\sb{i}:\DK(\Gamma)\to\cS} \wb{i=1,2} which agree on
\w[,]{\sk{n-1}\DK(\Gamma)} (or more generally, where we have a given weak
equivalence between the two restrictions. This difference obstruction is just
the \wwb{n-2}nd $k$-invariant for resulting prismatic diagram, as follows:
\end{mysubsection}

\begin{example}\label{egprism}
When \w{\Gamma=\bone{2}} is a square, the resulting prism is a cube, as in
\wref[,]{eqprism} with the top and bottom facets given by \w{\widehat{X}\sb{1}}
and \w{\widehat{X}\sb{2}} respectively, and the four vertical facets represent
the chosen equivalence on the $1$-skeleton (which need not be the identity).
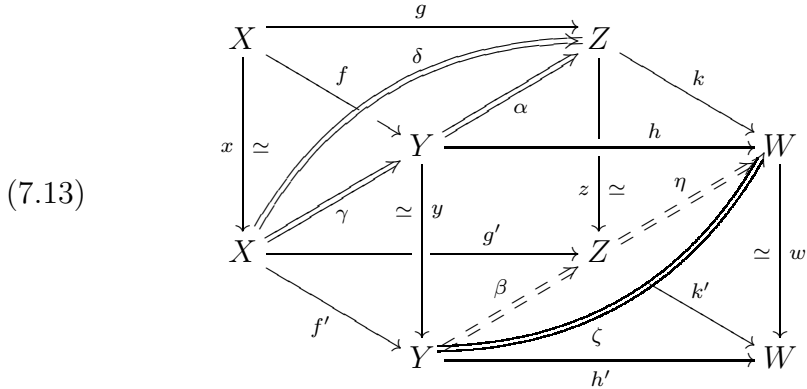
\myudiag[\label{eqprism}]{
%
%
X \ar@<1ex>[rr]\sp{g} \ar|(.7){\hole}[rd]\sp{f}
\ar[dd]\sb{x}\sp{\simeq} &&
Z \ar[rd]\sp{k} \ar|(.5){\hole}[dd]\sb(0.7){z}\sp(0.7){\simeq} &&  \\
& Y \ar[rr]\sp(0.65){h} \ar[dd]\sp(0.3){y}\sb(0.3){\simeq}
\ar@{=>}[ru]\sb{\alpha} &&
W \ar[dd]\sp{w}\sb{\simeq} \\
%
%
X \ar|(.5){\hole}[rr]\sp(0.7){g'} \ar[rd]\sb{f'}
\ar@{=>}[ru]\sb{\gamma} \ar@/^2.1em/@{==>}[rruu]\sp(0.6){\delta}
&& Z \ar|(.23){\hole}[rd]\sp{k'} \ar@{==>}[ru]\sp{\eta} & \\
& Y \ar[rr]\sb{h'} \ar@{==>}[ru]\sp{\beta}
\ar@<1ex>@/_2.1em/@{=>}[rruu]\sb(0.4){\zeta} && W
}
\noindent We see that \w{\widehat{X}\sb{1}} and \w{\widehat{X}\sb{2}}
are not equivalent in \w{\cS\sp{\Gamma}} if and only if the obstruction $\beta$
for \wref{eqprism} (defined as in Example \ref{egboundcube}) does not vanish
for any choice of the vertical facets.
\end{example}

%
%
\sect{Obstruction Theory for the Realization of \ww{\EE{2}}-algebras}
\label{cotrea}

In this section, we will apply our homotopy limit decomposition techniques
to study another example in more detail: the extension of
an \ww{\EE{1}}-structure on a space $\bX$  to an \ww{\EE{2}}-structure.

\begin{defn}\label{dmonoid}
We denote by \w{\FreeBrCat} the $(2, 1)$-category of free braided monoidal categories
on $n$ generators. By thinking of \w{\FreeBrCat} as a category enriched in categories,
we can apply the nerve construction to each mapping space to obtain a simplicial
category, which we denote by \w[.]{\mFreeBrCat}
We will denote by \w{\FreeMon} and \w{\mFreeMon} the $1$-category of free monoids,
regarded respectively as a horizontally discrete $2$-category and a discrete
simplicial category.
\end{defn}

\begin{example}\label{exam10.1}
The \emph{braid category} $\B$ is the category which has objects the natural numbers
$\bN$, with arrows \w{n \to n} given by the braid group \w{B\sb{n}}
(see \cite{MCat}). There are no morphisms \w{m \to n} for \w[.]{m \neq n}
The \emph{free braided monoidal category on one generator} $\bB$ (described in
\cite[XI.5]{MCat}) has a monoidal product given by addition on $1$-morphisms, and
acts on morphisms by braid addition
\w[.]{b\sb{m, n} : B\sb{m} \times B\sb{n} \to B\sb{m+n}}.
The braiding  \w{\gamma\sb{m, n} : n + m \to m+n} is given by crossing $m$
strings over $n$ strings.

We will write \w{\bB\sp{n}}  for the free braided monoidal category on $n$ generators.
It is just \w[,]{\prod\sb{i=1}\sp{n} \bB } and on underlying categories \w[.]{\NN\sp{m}}
The morphisms
\w{\bB\sp{n} \to \bB} are determined by choices of $n$-tuples
\w{(x\sb{1}, \cdots, x\sb{n}) \in \NN\sp{n}} and permutations
\w[,]{\sigma \in S\sb{n}} and are given by the formula
\begin{myeq}\label{mult1}
(m\sb{1}, \cdots m\sb{n})~\mapsto~\Sigma\sb{i=1}\sp{n} x\sb{i} \cdot m\sb{\sigma(i)}
\end{myeq}
\noindent Natural transformations exist only between morphisms determined by
\w{((x\sb{1}, \cdots, x\sb{n}), \sigma)} and \w[,]{((x\sb{1}, \cdots, x\sb{n}), \sigma')}
  and are generated by
\begin{myeq}\label{nat}
\Sigma\sb{i=1}\sp{j-1} id\sb{x\sb{\sigma(i)}} \cdot m\sb{\sigma(i)} +
\gamma\sb{x\sb{\sigma(j)} \cdot m\sb{\sigma(j)}, x\sb{\sigma(j+1)} \cdot
m\sb{\sigma(j+1)}} + \Sigma\sb{i=j+2}\sp{n} id\sb{x\sb{\sigma(i)}} \cdot m\sb{\sigma(i)}
\end{myeq}
\end{example}

In \cite{HMeadS}, the authors established an equivalence of $\infty$-categories
$$
\mu\sp{(-)}:\mMnd\sb{\Fin}(\cS)\leftrightarrows \Th\sb{\Fin}
$$
between $\Fin$-nervous monads and \ww{\Fin}-theories, the latter of which are certain
essentially
surjective functors of $\infty$-categories \w[.]{\Fin \to \A}
There is also an equivalence of $\infty$-categories:
$$
\Model(\Th(M)) \simeq\Alg\sb{M}
$$
between \ww{\Th(M)}-models, where \w{\Model(\Th(M))} are simply finite product
preserving functors \w[,]{F :\Th(M) \to\cS} and algebras for the monad $M$.
In essence, this allows one to describe algebras of monads in simpler terms as models
of theories. The case of the free operad monads is discussed in \cite[Section 8]{HMeadS}.

We can use the description of operad algebras as models of theories to develop
an obstruction theory for \ww{\EE{2}}-algebras.
For any space $\bX$, let \w{(\sSet)\sb{\bX}} be the full simplicial subcategory
of \w{\sSet} with objects the cartesian products \w{\bX\sp{n}} \wb[.]{n=1,2,\dotsc}

\begin{prop}\label{lem10.2}
An \ww{\EE{1}}-algebra structure on a space $\bX$ is equivalent to
a functor  \w{\psi:\DK(\mFreeMon\op)\to\sSet} sending the object $1$ to $\bX$.
Extending this to an \ww{\EE{2}}-algebra structure is the same as
finding an extension
\myudiag[\label{eqlifting}]{
\DK(\mFreeMon\op) \ar[r]\sb{\psi} \ar[d]\sb{\Phi} & (\sSet)\sb{\bX} \\
\DK(\mFreeBrCat\op)~, \ar@{.>}[ur] &
}
\noindent for $\Phi$ as in Example \ref{exam10.1}.
\end{prop}

\begin{proof}
This follows from the description of \w{\mFreeMon} and
\w{\mFreeBrCat} as the pre-theories describing \ww{\EE{1}}- and
\ww{\EE{2}}-spaces, respectively (see \cite[Examples 8.12,8.15]{HMeadS}).
Note that any extension to \w{\DK(\mFreeBrCat\op)} is automatically
product-preserving, since $\Phi$ is.
\end{proof}

See \cite{FWilH} for another approach to a variant of our problem.

\begin{remark}\label{myrmk}
We get the extension by lifting along the Postnikov tower for
\w{(\sSet)\sb{\bX}} as a simplicial category. Composition in
\w{\DK(\mFreeBrCat)} is a monomorphism in each simplicial degree,
so we can apply Theorem \ref{thm5.8} to write \w{\DK(\mFreeBrCat)}
as an homotopy colimit of $n$-cubes, and use this description to
identify the obstruction to lifting. In particular, if for some
boundary of an cube the obstruction does not vanish, this
precludes the existence of a lift to the next level.

Note that an \ww{\EE{1}}-algebra structure on a space $\bX$
is equivalent to giving it the structure of a strictly associative monoid (see
\cite[Theorem 13.4]{MayG}), so we may think of $\bX$ as a simplicial category $\mX$
with a single object. Because the monoidal structure in
\w{\sSet} is Cartesian, we see that we do not need the full category
\w{(\sSet)\sb{\bX}} in \wref[,]{eqlifting} but only those components of the mapping
spaces of \w{\sSet} corresponding to the categorical and product structure maps
of $\mX$.
\end{remark}

\begin{mysubsection}{The first obstruction}
\label{sfo}
Let \w{\Phi\colon \DK(\mFreeBrCat) \to \sSet} be a morphism which determines an
\ww{\EE{2}}-algebra structure on a space $\bX$. By the description of \w{\DK} in
\cite{DKanL}, \w{\DK(\mFreeBrCat)\sb{0}}  is freely generated by morphisms of the form
 \w[,]{\NN\sp{m} \to \NN\sp{k}} determined by the universal property of products
 by maps of the form \wref[.]{mult1} Maps of the form
 \w{(m\sb{1}, \cdots, m\sb{n})\colon\NN\sp{n} \to \NN} are taken to $n$-fold
multiplication maps
$$
\bX\sp{n} \to \bX, (x\sb{1}, \cdots, x\sb{n}) \mapsto x\sb{1}\sp{m\sb{\sigma(1)}}
\cdots x\sb{n}\sp{m\sb{\sigma(n)}}~,
 $$
\noindent associative up to homotopy.

By the description of \w{\DK} in \cite{DKanL}, the morphisms of
\w{\DK(\mFreeBrCat)\sb{1}} are of the form
$$
((\eta\sb{1}\sp{1}) \cdots (\eta\sb{n}\sp{k\sb{1}})) \cdots
((\eta\sb{n}\sp{1}) \cdots (\eta\sb{n}\sp{k\sb{m}}))
$$
where \w{\eta\sb{i}\sp{j}} is generated by morphisms of the form \wref[,]{nat}
which is taken to a natural transformation of maps
\w{\psi\sb{1}, \psi\sb{2} : \bX\sp{n} \to \bX} given by
\begin{myeq}\label{eqxx}
\begin{split}
\psi\sb{1}(x\sb{1} \cdots, x\sb{n})~=&~ x\sb{\sigma(1)}\sp{m\sb{\sigma(1)}} \cdots
  x\sb{\sigma(j)}\sp{m\sb{\sigma(j)}} x\sb{\sigma(j+1)}\sp{m\sb{\sigma(j+1)}} \cdots
  x\sb{n}\sp{m\sb{\sigma(n)}}\\
\psi\sb{2}(x\sb{1} \cdots, x\sb{n})~=&~
  x\sb{\sigma(1)}\sp{m\sb{\sigma(1)}} \cdots x\sb{\sigma(j+1)}\sp{m\sb{\sigma(j+1)}}
  x\sb{j}\sp{m\sb{\sigma(j)}}  \cdots x\sb{n}\sp{m\sb{\sigma(n)}}~.
\end{split}
\end{myeq}
\noindent Thus, a morphism \w{((\eta\sb{i}\sp{j}))} is taken to a morphism that commutes
the factors in an ($n$-fold) multiplication. In order to construct this, we may
choose a unique natural transformation \w{s : y \cdot x \simeq x \cdot y} and
take \wref{eqxx} to be \w[.]{\Id\sp{j-1} \times s \times \Id\sp{n-j-1}}

As an example, a typical square in the image of \w{\DK(\mFreeBrCat)} is given by
$$
\xymatrix
{
  \bX\sp{3}  \ar[d]\sb{\Id \times m}
  \ar[rr]\sp{(m\sigma)\times\Id}
  && \bX\sp{2} \ar[d]\sp{m} \\
\bX\sp{2} \ar@{=>}[rru]\sb{\gamma} \ar[rr]\sb{m} &&  \bX
 }
$$
where \w{\sigma:\bX\sp{2}\to \bX\sp{2}} switches the factors and $m$ is
multiplication. Here $\gamma$ is the composite
\w[,]{a \circ m(s \times\Id)} for \w{s : m\sigma \simeq m} a chosen braiding
(commuter), and \w{a : x \cdot (y \cdot z) \simeq (x \cdot y) \cdot z}
an associator.

In this language, the three-dimensional cube in the image of \w{\DK(\mFreeBrCat)}
which describes the first obstruction class is given by
\myrdiag[\label{eqglobssp}]{
\bX\sp{3} \ar[rr]\sp{\tau\sp{0}} \ar[rd]\sp(0.6){\tau\sp{1}} \ar[dd]\sb{\tau\sp{2}} &&
  \bX\sp{3} \ar[rd]\sp{p} \ar|(.5){\hole}[dd]\sb(0.2){q}& \\
  & \bX\sp{3} \ar[rr]\sp(0.3){p} \ar[dd]\sb(0.25){q} && \bX\sp{2} \ar[dd]\sp{m}\\
\bX\sp{3}  \ar|(.5){\hole}[rr]\sb(0.7){q} \ar[rd]\sb{p} && \bX\sp{2} \ar[rd]\sp{m} & \\
& \bX\sp{2}  \ar[rr]\sb{m} && \bX,
}
\noindent where $\tau$ is the cyclic shift \w[,]{(a,b,c)\mapsto(b,c,a)}
$p$ multiplies the first two entries, $q$ multiplies the last two entries and
transposes the resulting pair, and $m$ is the multiplication.

In order to abstract from \wref{eqglobssp} a more manageable ``local'' obstruction
to having an \ww{\EE{2}}-algebra structure on $\bX$, we think of it as a
simplicial category $\mX$ with a single object $\star$ (see \S \ref{myrmk}): thus points
in $\bX$ (thought of as a loop space) are represented as arrows in $\mX$, and the group
structure is given by composition. A point in the initial vertex of the cube
\wref{eqglobssp} then corresponds to three arrows \w[,]{e,f,g:\star\to\star}
and tracking the various compositions there reduces to making choices of
braidings \w[,]{\beta:fg\sim gf} and so on, in
\myxdiag[\label{eqcpt}]{
  \star\ar[rr]\sp{g} \ar[rd]\sp(0.6){f} \ar[dd]\sb{e} &&
  \star \ar[rd]\sp{f} \ar|(.5){\hole}[dd]\sp(0.7){e}& \\
& \star\ar[rr]\sp(0.3){g} \ar[dd]\sb(0.25){e} && \star \ar[dd]\sp{e}\\
\star \ar|(.5){\hole}[rr]\sb(0.7){g} \ar[rd]\sb{f} && \star\ar[rd]\sp{f} & \\
& \star \ar[rr]\sb{g} && \star
}
\noindent (the choices should be the same for parallel facets).

Now recall that for any $H$-group $\bX$ and pointed map
\w[,]{\alpha\times\beta:\bA\times\bB\to\bX\times\bX} the corresponding
\emph{Samelson product} is defined to be the map
\w{\bA\wedge\bB\to\bX} induced by the commutator map \w{[-,-]:\bX\times\bX\to\bX}
(see \cite[Ch.\ X, \S 5]{GWhE}):
travelling along the boundary of each $2$-dimensional
facet of \wref{eqcpt} in a fixed direction yields a commutator \wh e.g.,
\w{fgf\sp{-1}g\sp{-1}} along the top square \wh and thus a Samelson product in
$\bX$ for \w[.]{\bA=\bB=\bS{1}}

To identify the map described by the boundary of the cube \wref{eqcpt} as a whole,
consider the cofibration sequence
$$
\bS{2}~\xra{h}~\bT{1}~\xra{\iota}~\bS{1}\times\bS{1}\times\bS{1}~\to~\bS{3}~,
$$
\noindent with the fat wedge \w{\bT{1}=\bT{1}(\bS{1},\bS{1},\bS{1})} of
\wref{eqfatwedge} corepresenting initial data consisting of
three given maps \w[,]{\alpha:\bS{1}\to\bX} \w[,]{\beta:\bS{1}\to\bX} and
\w{\gamma:\bS{1}\to\bX} into a loop space \w{\bX=\Omega\bY}
with all pairwise Samelson products vanishing. Here we specialize
to \w[,]{\bA=\bB=\bC=\bS{1}} though this may be defined more generally
(see \cite[\S 6]{GPorHP}).

Let \w{\varphi:\bS{1}\times\bS{1}\times\bS{1}\to\bY\sp{\bT{1}}} cover
\w{\alpha\times\beta\times\gamma} under
\w{\iota\sp{\ast}:\bY\sp{\bT{1}}\to\bY\sp{\bS{1}\vee\bS{1}\vee\bS{1}}=\bX\sp{3}}
(for $\iota$ the inclusion). The \emph{secondary Samelson product}
\w{\lra{\alpha,\beta,\gamma}\in[\bS{1}\wedge\bS{1}\wedge\bS{1},\bX]}
is then defined to be \w[,]{q\circ h\sp{\ast}(\varphi)}
where
$$
q:[\bS{1}\times\bS{1}\times\bS{1},\Omega\bY]=[\Sigma(\bS{1}\times\bS{1}\times\bS{1}),\bY]
~\to~
[\Sigma(\bS{1}\wedge\bS{1}\wedge\bS{1}),\bY]=[\bS{1}\wedge\bS{1}\wedge\bS{1},\Omega\bY]
$$
\noindent is the projection under the standard splitting (see
\cite[Ch.\ X, (8.20)]{GWhE}). By \cite[Theorem 6.12]{GPorHP}, it
corresponds to the secondary Whitehead product under the loop-suspension adjunction.

Thus \w{h\sp{\ast}:\bY\sp{\bT{1}}\to\Omega\sp{2}\bY} is universal for
the secondary Samelson product: it is null-homotopic if $\bY$ is an $H$-space,
by \cite[Theorem 6.4]{GPorHP} \wh so in particular, if $\bX$ is an
\ww{\EE{2}}-space (compatibly with the given loop structure).
We deduce that the obstruction class $\beta$ of Theorem \ref{tpbclass} corresponding to
\wref{eqglobssp} can be thought of as the \emph{global secondary Samelson product},
which specializes to the usual secondary Samelson product
for each choice of points $e$, $f$, and $g$ in $\bX$.
\end{mysubsection}

\begin{examples}
\label{egcpt}
Consider the homotopy-abelian loop space
$$
\bX=\Omega\bY=\Po{4}\Omega\CP{2}\hs\text{for}\hs
\bY:=\Po{5}\CP{2}
$$
(see \cite[3.10]{BGanN}): the non-trivial secondary Samelson
product in \w{\pi\sb{4}\bX\cong\pi\sb{5}\CP{2}} (see \cite[Corollary 2]{GPorHO}) is
thus an obstruction to $\bX$ being an \ww{\EE{2}}-space.

Another example of a homotopy-abelian $H$-space which has an \ww{\EE{1}}- but not
an \ww{\EE{2}}-algebra structure is \w[,]{\Omega\CP{3}} which is homotopy-abelian
(see \cite[Theorem 1.18]{StasHA}) \wh in fact, it is $H$-equivalent to
\w{\bS{1}\times\Omega\bS{7}} \wwh but is not an \ww{\EE{2}}-space even rationally
(see \cite{ACuB}). In this case the obstruction is the
third-order Whitehead product in \w[.]{\pi\sb{7}\CP{3}}
\end{examples}

%
%

\end{document}